\documentclass[11pt,letterpaper]{amsart}

\usepackage[dvipsnames]{xcolor}
\usepackage[a4paper,top=3cm,bottom=2.5cm,left=3cm,right=3cm,marginparwidth=1.75cm]{geometry}

\usepackage{amsmath, amssymb, amsfonts, latexsym, mdwlist, amsthm}
\usepackage{subfig}
\usepackage{graphicx}
\usepackage{tikz-cd}
\usepackage{wrapfig}

\usepackage{mathtools}
\usepackage{tikz}

\usetikzlibrary{patterns}

\usepackage{enumerate}

\usepackage{amsmath, amssymb, amsfonts, latexsym, mdwlist, amsthm}
\usepackage{subfig}
\usepackage{wrapfig}

\usepackage{mathtools}
\usepackage{bm}

\definecolor{lightred}{HTML}{ff4d4d}
\definecolor{lightblue}{HTML}{1F88CD}
\definecolor{lightgrey}{HTML}{727272}
\definecolor{lightblue2}{HTML}{009EC1}
\definecolor{mypink}{HTML}{FD00B0}

\usepackage[bookmarks, colorlinks, breaklinks, pdftitle={},
pdfauthor={Zhiyu Liu}]{hyperref}
\hypersetup{linkcolor=OliveGreen,citecolor=MidnightBlue,filecolor=black,urlcolor=blue}

\usepackage{tikz}
\usetikzlibrary{calc,trees,positioning,arrows,chains,shapes.geometric,%
    decorations.pathreplacing,decorations.pathmorphing,shapes,%
    matrix,shapes.symbols}

\tikzset{
>=stealth',
  punktchain/.style={
    rectangle,
    rounded corners,
    draw=black, thick,
    minimum height=3em,
    text centered,
    on chain},
  line/.style={draw, thick, <-},
  element/.style={
    tape,
    top color=white,
    bottom color=blue!50!black!60!,
    minimum width=8em,
    draw=blue!40!black!90, very thick,
    text width=10em,
    minimum height=3.5em,
    text centered,
    on chain},
  every join/.style={->, thick,shorten >=1pt},
  decoration={brace},
  tuborg/.style={decorate},
  tubnode/.style={midway, right=2pt},
}

\usepackage{paralist}
\setdefaultenum{(a)}{(i)}{}{}
\usepackage[shortlabels]{enumitem} 

\makeatletter
\newtheorem*{rep@theorem}{\rep@title}
\newcommand{\newreptheorem}[2]{%
\newenvironment{rep#1}[1]{%
 \def\rep@title{#2 \ref{##1}}%
 \begin{rep@theorem}}%
 {\end{rep@theorem}}}
\makeatother

\newtheorem{theorem}{Theorem}[section]
\newreptheorem{theorem}{Theorem}
\newtheorem{proposition}[theorem]{Proposition}

\newtheorem{lemma}[theorem]{Lemma}

\newtheorem{corollary}[theorem]{Corollary}
\newreptheorem{corollary}{Corollary}
\newtheorem{conjecture}[theorem]{Conjecture}
\newreptheorem{conjecture}{Conjecture}

\newtheorem{thm-int}{Theorem}

\theoremstyle{definition}
\newtheorem{Def-s}[theorem]{Definition}
\newtheorem{definition}[theorem]{Definition}

\newtheorem{remark}[theorem]{Remark}

\newcommand{\ignore}[1]{}

\usepackage[numbers]{natbib}
\newcommand{\ra}{\rightarrow}
\newcommand{\xra}{\xrightarrow}

\newcommand{\sst}{\subset}

\newcommand{\D}{\mathrm{D}}

\newcommand{\ZZ}{\mathbb{Z}}

\newcommand{\QQ}{\mathbb{Q}}

\newcommand{\CC}{\mathbb{C}}
\newcommand{\PP}{\mathbb{P}}

\newcommand{\DT}{\mathsf{DT}}
\newcommand{\GW}{\mathsf{GW}}
\newcommand{\PT}{\mathsf{PT}}

\newcommand{\Hilb}{\mathrm{Hilb}}
\newcommand{\FHilb}{\mathrm{FHilb}}

\newcommand{\ch}{\mathrm{ch}}

\newcommand{\KK}{\mathrm{K}}

\newcommand{\pr}{\mathrm{pr}}

\renewcommand{\Re}{\operatorname{Re}}
\renewcommand{\Im}{\operatorname{Im}}

\DeclareMathOperator{\rk}{rk}
\DeclareMathOperator{\cok}{cok}

\DeclareMathOperator{\Coh}{\mathrm{Coh}}

\DeclareMathOperator{\sF}{\mathsf{F}}

\DeclareMathOperator{\Hom}{Hom}

\DeclareMathOperator{\Pic}{Pic}

\newcommand{\cC}{\mathcal{C}}
\newcommand{\cA}{\mathcal{A}}

\newcommand{\cH}{\mathcal{H}}

\newcommand{\cM}{\mathcal{M}}

\DeclareMathOperator{\oh}{\mathcal{O}}

\usepackage{hyperref}
\hypersetup{
	colorlinks=true,
    linkcolor={blue},
    citecolor={lightred},
	urlcolor={black}
}

\begin{document}

\title[Castelnuovo bound and GW-invariants of quintic 3-folds]{Castelnuovo bound and higher genus Gromov--Witten invariants of quintic 3-folds}
\subjclass[2020]{14N35 (Primary); 14J33, 14H50, 14F08 (Secondary)}
\keywords{Castelnuovo bound, Gromov--Witten invariants, Gopakumar--Vafa invariants, Mirror symmetry, BCOV theory.}

\author{Zhiyu Liu}

\address{Institute for Advanced Study in Mathematics, Zhejiang University, Hangzhou, Zhejiang Province 310030, P. R. China}
\address{College of Mathematics, Sichuan University, Chengdu, Sichuan Province 610064, P. R. China}
\email{zhiyuliu@stu.scu.edu.cn, jasonlzy0617@gmail.com}
\urladdr{sites.google.com/view/zhiyuliu}

\author{Yongbin Ruan}

\address{Institute for Advanced Study in Mathematics, Zhejiang University, Hangzhou, Zhejiang Province 310030, P. R. China}
\email{ruanyb@zju.edu.cn}


\begin{abstract}

We prove a conjectural vanishing result for Gopakumar--Vafa invariants of quintic 3-folds, referred to as Castelnuovo bound in the literature. Furthermore, we calculate Gopakumar--Vafa invariants at Castelnuovo bound $g=\frac{d^2+5d+10}{10}$. As physicists showed, these two properties allow us to compute all Gromov--Witten invariants of quintic 3-folds up to genus $53$ \cite{HKQ09,PhysRevD.79.066001}, provided that the conifold gap condition holds. We also give a bound for the genus of any one-dimensional closed subscheme in a smooth hypersurface of degree $\leq 5$, which may be of independent interest.

\end{abstract}

\vspace{-1em}
\maketitle

\setcounter{tocdepth}{1}
\tableofcontents

\section{Introduction}

\subsection{Background}

One of the most difficult problems in geometry and physics is to compute higher genus Gromov--Witten (GW) invariants of compact Calabi--Yau 3-folds (symbolized by the famous quintic 3-fold which we will focus on in this article). In the 90s, the effort to compute genus zero invariants of Calabi--Yau 3-folds led to the birth of mirror symmetry as a mathematical subject! Afterward, it was quickly realized that the higher genus computation is much more difficult.
It took another ten years for Zinger to compute
genus one GW-invariants \cite{Zi08} of quintic 3-folds and yet 
another ten years for the second author and his collaborators to compute genus two case \cite{GJR17P}. More than ten years ago, physicists Klemm and his group shocked the community by announcing a physical derivation of GW-invariants of quintic 3-folds up to genus 53 \cite{HKQ09,PhysRevD.79.066001}. This was the time when mathematicians could only do the calculation for genus 1. The record set by physicists seemed to be unreachable by a mathematician in one's lifetime.

Klemm group's physical derivation used five mathematical conjectures: Four BCOV axioms (finite generation, holomorphic anomaly equation, orbifold regularity and conifold gap condition) and Castelnuovo bound. The BCOV axioms originated from B-model and can be viewed as a set of conjectural axioms to describe the structure of the higher genus GW generating series. To simplify the notation, let us restrict ourselves to a quintic 3-fold $X$. Let $N_{g,d}$ be the genus $g$ degree $d$ GW-invariant of $X$.  We define the genus $g$ GW generating series 

$$F_g(t):=\sum_{d\geq 0} N_{g, d}  t^{d}.$$

\begin{enumerate}[(1)]

\item The finite generation axiom asserts that

$$F_g=Y^{-(g-1)}P_g (\chi_1,\chi_2,\chi_3,\chi_4, Y)$$
for a polynomial $P_g$ where $\chi_1,\chi_2,\chi_3,\chi_4$ and $Y$ are explicit generating series constructed out of genus zero data;

\item Holomorphic anomaly equation asserts that $\frac{\partial P_g}{\partial \chi_i}$ for $1\leq i\leq 4$ can be explicitly expressed in terms of lower genus data. Therefore, the monomials of $P_g$ containing $\chi_i$ are determined by lower genus data. Moreover, 
$$f_g:=P_g(0,0,0,0,Y)=\sum_{i=0}^{3g-3} a_{i,g}Y^i$$
is a polynomial in $Y$ of degree $3g-3$;

\item The orbifold regularity asserts that 
$$a_{i,g}=0$$
for $i\leq \lceil\frac{3g-3}{5}\rceil$;

\item The conifold gap condition determines $a_{i,g}$ for $i\geq g$ recursively from lower genus data. 
\end{enumerate}

After applying all BCOV axioms, we only need to determine $\lfloor\frac{2g-2}{5}\rfloor$ many initial conditions (i.e.~coefficients $\{a_{i,g}\}_{i=\lceil\frac{3g-3}{5}\rceil+1}^{g-1}$) for $Y$ to derive $F_g(t)$. Then the \emph{Castelnuovo bound}, roughly speaking, fixes a large number of remaining initial conditions, which allows us to compute $F_g$ for $g\leq 53$. It first appeared in the literature in physics (e.g.~\cite{KKV,HKQ09} and \cite[(4.1)]{PhysRevD.79.066001}), which predicts the vanishing of Gopakumar--Vafa invariants for a given degree at sufficiently high genus.

During the last few years, dramatic progress has been achieved to prove these conjectures. For the quintic 3-fold, based on theories developed in \cite{Chen_2021,Chang_2021,CJR22,CJRS18}, the finite generation and holomorphic anomaly equation have been proved independently by Guo--Janda--Ruan \cite{GJR18} and Chang--Guo--Li \cite{CGL18,CGLFeynman}. The orbifold regularity was proved by Guo--Janda--Ruan \cite{GJR18}. The conifold gap condition was proved by Guo--Janda--Ruan for $g\leq 5$ \cite{GJR18}. As a consequence, they computed GW-invariants up to $g=5$. The main content of this article is to prove the last conjecture, i.e., Castelnuovo bound.

\subsection{The main theorems}

To state our results, we need to introduce Gopakumar--Vafa (GV) invariants. In physics, these invariants count the number of BPS states (c.f.~\cite{GV1,GV2}). To define Gopakumar--Vafa invariants, let

$$\GW(\lambda, t)=\sum_{d>0}\sum_{g\geq 0}N_{g,d}\lambda^{2g-2}t^d.$$

Then, we expand it in terms of the Fourier series

\begin{equation} \label{GW=GV}
    \GW(\lambda, t)=\sum_{g\geq 0} \sum_{d>0} \sum_{r\geq 1} \frac{n^d_g}{r}\cdot \big(2\sin(\frac{r\lambda}{2})\big)^{2g-2}\cdot t^{rd}.
\end{equation}

We define \emph{Gopakumar--Vafa (GV) invariants} of $X$ to be the rational numbers $n_g^{d}$. The GW and GV-invariants are uniquely determined by each other.
From the above definition, it is not clear that $n_{g}^d$ are integers. The integrality of $n_g^d$ was proved by Ionel--Parker \cite{IP13} via the symplectic geometric definition of $n_g^d$.

Now, we can state our main results.

\begin{theorem} \label{main_thm_1.1}
Let $X$ be a quintic 3-fold. Then we have
\[n^d_g=0\]
for any
\[g>\frac{d^2+5d+10}{10}.\]
\end{theorem}

It is easy to check that the line $d=\frac{2g-2}{5}$ and the conic $g=\frac{d^2+5d+10}{10}$ intersect at $(g=1,d=0)$ and $(g=51,d=20)$. Moreover, $g=\frac{d^2+5d+10}{10}$ bounds $d=\lfloor\frac{2g-2}{5}\rfloor$ for $g\leq 53$. Hence,

\begin{corollary} \label{cast_bound}

Suppose that $g\leq 53$ and $g\neq 51$. Then we have
$$n^d_g=0$$
for $d\leq \lfloor\frac{2g-2}{5}\rfloor$. 
Furthermore, we have
$$n^d_{51}=0$$
for $d\leq 19$.

\end{corollary}

Using Corollary \ref{cast_bound}, it is enough to compute GW-invariants of quintic 3-fold for $g\leq 50$. 
In order to compute GW-invariants up to $g=53$, we need another non-vanishing result, which also shows that the Castelnuovo bound (Theorem \ref{main_thm_1.1}) is sharp for $d\in 5\ZZ$. 

\begin{theorem}\label{non_vanish}
Let $m\geq 2$ and $d=5m$ be integers. Then we have
\[n_{\frac{d^2+5d+10}{10}}^d=(-1)^{\binom{m+3}{3}-\binom{m-2}{3}+3}\cdot 5\big(\binom{m+3}{3}-\binom{m-2}{3}\big).\]
Here if $m-2<3$, we set $\binom{m-2}{3}:=0$.
\end{theorem}

\begin{remark}
Applying Corollary \ref{cast_bound}, BCOV axioms and the case $m=4$ of Theorem \ref{non_vanish}, physicists Klemm and his group derived $F_g(t)$ for $g\leq 53$ \cite{HKQ09} (see Appendix \ref{sec_appendix}).
\end{remark}

\begin{remark}
The above theorem \ref{non_vanish} is predicted in \cite{PhysRevD.79.066001} for $2\leq m\leq 4$. Our theorem \ref{non_vanish} also covers the case of $m\geq 5$ where $g=\frac{d^2+5d+10}{10}\geq 76$. For the case $m=1$, we have $n^5_6=10$ (see Remark \ref{m=1}).
\end{remark}

\begin{remark}
There are various algorithms to compute higher genus GW-invariants, e.g.~\cite{MP06, FanLee19}. However, these algorithms grow exponentially in complexity and it is difficult to be put into computers to calculate specific invariants. Some computerized results of higher genus invariants using BCOV axioms and Castelnuovo bound (Theorem \ref{main_thm_1.1}) can be found in \cite[Table 2, Table 3]{HKQ09}.
\end{remark}

We will prove several more general results focusing on the bound for the genus of curves in a smooth projective 3-fold that satisfies the conjectural Bogomolov--Gieseker inequality of \cite{bayer2011bridgeland,bayer2016space}.

First, we fix some notations. If $X$ is a smooth projective 3-fold of Picard rank one, we denote the first Chern class of the ample generator of $\Pic(X)$ by $H$. Then we define the degree and index of $X$ by $n:=H^3$ and $K_X=-iH$, respectively.

\begin{theorem}[{Theorem \ref{BMT_thm}}]
Let $X$ be a smooth projective 3-fold of Picard rank one of degree $n$ and index $i$. Assume that $X$ satisfies Conjecture \ref{SBG_conj} for all $(a,b)\in \mathbb{R}_{>0}\times \mathbb{R}$. Then for any one-dimensional closed subscheme $C\subset X$ of degree $d$ and genus $g$, we have
\[g\leq \frac{1}{2n}d^2+\frac{1-i}{2}d+1.\]
\end{theorem}

For example, the assumptions in the above theorem hold for any Fano 3-fold or abelian 3-fold of Picard rank one (see Theorem \ref{SBG}).


\begin{remark}
We should mention that Theorem \ref{BMT_thm} was proved in \cite[Proposition 1.3]{MS20} with an extra assumption that $\ch_2(E)\in \frac{1}{2}H^2\cdot \ZZ$ and $\ch_3(E)\in \frac{1}{6}H^3\cdot \ZZ$ for any $E\in \Coh(X)$ (see \cite[Assumption A]{MS20}). However, as Soheyla Feyzbakhsh pointed out for us that this assumption has not been used in their proof critically, thus Theorem \ref{BMT_thm} has been proved by \cite{MS20} indeed. After the first version of our paper was completed, we received a note \cite{F2022} from her, in which she also sketches a proof of Theorem \ref{BMT_thm} without this extra assumption.
\end{remark}

For a smooth hypersurface $X\subset \PP^4$, it is known that $X$ satisfies the assumption in Theorem \ref{BMT_thm} if $\deg X\leq 4$. However, Conjecture \ref{SBG_conj} only holds when $(a,b)$ is in a restricted region for quintic 3-folds (see Theorem \ref{SBG}), so we can not apply Theorem \ref{BMT_thm} directly. However, we can even get a better bound.

\begin{theorem}[{Theorem \ref{bound_quintic}}]
Let $X\subset \PP^4$ be a smooth hypersurface of degree $n\leq 5$. Then for any one-dimensional closed subscheme $C\subset X$ of degree $d$ and genus $g$, we have
\[g\leq \frac{1}{2n}d^2+\frac{n-4}{2}d+1.\]
More importantly, when $C$ is not contained in any hyperplane section of $X$, we have
    \[g\leq \frac{1}{2n}d^2+(\frac{n-4}{2}-\frac{1}{n})d+2+\frac{1}{n}.\]
\end{theorem}

When $n=1$, i.e.~the case $X=\PP^3$, our theorem recovers Hartshorne's classical result \cite{Hartshorne1994TheGO}.

Moreover, we can describe subschemes of the maximal genus $g=\frac{1}{2n}d^2+\frac{n-4}{2}d+1$ completely.

\begin{theorem}[{Theorem \ref{20_51}}]
Let $X\subset \PP^4$ be a smooth hypersurface of degree $n\leq 5$. Let $C$ be a one-dimensional closed subscheme $C\subset X$ of degree $d>n$ and genus $g$. Then \[g=\frac{1}{2n}d^2+\frac{n-4}{2}d+1\]
if and only if $C$ is a complete intersection of $X$ with $\PP^3$ and a degree $\frac{d}{n}$ hypersurface in $\PP^4$.
\end{theorem}

Using Theorem \ref{20_51}, we can describe the corresponding Hilbert schemes explicitly in Proposition \ref{moduli_fibre}.

Note that most of the results in Section \ref{sec_bound} are true in positive characteristic as well (Remark \ref{char_p}) and our techniques can work in more general settings (Remark \ref{main_rmk}).

\subsection{Method}

Our approach is based on the Pandharipande--Thomas formulation of GV-invariants from PT theory. The Pandharipande--Thomas (PT) invariants of a quintic 3-fold $X$ are defined in \cite{PT07}, counting stable pairs in the derived category. For any pair of integers $n$ and $d>0$, the corresponding moduli space of stable pairs $P_n(X, d)$ is a projective scheme carries a zero-dimensional virtual cycle $[P_n(X, d)]^{vir}$. Then we define \emph{Pandharipande--Thomas invariants} by
\[P_{n,d}:=\int_{[P_n(X, d)]^{vir}} 1.\]
The generating series of PT-invariants is
\[\PT(q, t):=1+\sum_{d> 0}\sum_{n\in \ZZ} P_{n, d} q^n t^{d}.\]
We define the \emph{connected} series of PT-invariants by $\sF_{P}(q, t):=\log \PT(q, t)$. A geometric method to define these connected invariants is unknown yet.
Since the GW/PT correspondence \cite{PP12} holds for $X$ and we know that $n^{\beta}_g=0$ for any fixed $\beta$ and $g\gg 0$ by \cite{DIW21}, after changing the variable $q=-\exp(i\lambda)$ in \eqref{GW=GV}, we have
\begin{equation} \label{GV_PT_eq}
    \sF_{P}(q, t)=\sum_{g\geq 0} \sum_{d>0} \sum_{r\geq 1} n_g^{d} \frac{(-1)^{g-1}}{r} \big((-q)^{\frac{r}{2}}-(-q)^{-\frac{r}{2}}\big)^{2g-2} t^{rd}.
\end{equation}


Let us describe our strategy briefly.

First, we bound the genus of one-dimensional closed subschemes in $X$ (Theorem \ref{bound_quintic}). The main technique is wall-crossing of Bridgeland (weak) stability conditions and Bayer--Macr\'i--Toda's generalized Bogomolov--Gieseker inequality (Theorem \ref{SBG}). A naive observation is that if the ideal sheaf $I_C$ of a curve $C$ in $X$ has no wall in a special range, then the bound in Theorem \ref{bound_quintic} follows from applying Theorem \ref{SBG} to $I_C$ and specific $(a,b)$.


Unfortunately, there exist many walls for $I_C$ in general and the situation becomes more complicated. But a detailed analysis shows that either $C$ is in a hyperplane section or we can decompose $C$ as a union of subschemes $C=C_1\cup C_2$, each of them is one-dimensional and intersects with each other properly. Moreover, the length of $C_1\cap C_2$ can be controlled by wall-crossing. In the first case, the bound can be deduced from applying wall-crossing to some rank zero sheaves. In the latter case, using the formula
\[g(C)=g(C_1)+g(C_2)+\mathrm{length}(C_1\cap C_2)-1,\]
to bound $g(C)$ we only need to control $g(C_1)+g(C_2)$. This will be worked out by an induction argument combined with a bound for the degree of $C_2$ obtained from wall-crossing.

Now the first step gives the emptiness of moduli spaces $P_n(X,d)$, using which we can prove the vanishing of PT-invariants (Corollary \ref{PT_zero}) and \emph{connected} PT-invariants (Proposition \ref{logPT_zero}). Finally, a calculation of the generating series \eqref{GV_PT_eq} implies Theorem \ref{main_thm_1.1}.

For the non-vanishing result Theorem \ref{non_vanish}, we apply wall-crossing of weak stability conditions on $\PP^3$. The condition of a curve with the maximal genus implies the existence of certain walls, from which we can classify such curves (Theorem \ref{20_51}). Therefore, we can describe their moduli spaces explicitly (Proposition \ref{moduli_fibre}) and Theorem \ref{non_vanish} follows from computing the topological Euler characteristic of moduli spaces.

\subsection{Plan of the paper}
In Section \ref{sec:stability_condition}, we will review the general theory of stability conditions and wall-crossing.
The wall structure and explicit geometry of ideal sheaves of curves on 3-folds will be analyzed 
in Section \ref{wall_sec}. Then in Section \ref{sec_bound}, we will prove several bounds for the genus of curves on a class of smooth projective 3-folds (Theorem \ref{BMT_thm}). At the end of Section \ref{sec_4.1}, a bound for the genus of any one-dimensional closed subscheme in a smooth hypersurface of degree $\leq 5$ will be given (Theorem \ref{bound_quintic}). The main theorem  \ref{main_thm_1.1} will be proved in Section \ref{sec_vanish_GV}. The computation of  GV-invariants at Castelnuovo bound (Theorem \ref{non_vanish}) will be carried out in Section \ref{sec_nonvanish}. Since mathematical readers may not be familiar with the literature in physics, we include an appendix (due to Shuai Guo) to explain the reason that BCOV axioms together with the Castelnuovo bound and Theorem \ref{non_vanish} can compute $F_g$ up to $g\leq 53$.

\subsection*{Acknowledgements} 

The current work is completed in Institute for Advanced Study in Mathematics at Zhejiang University. Both authors thank the institute for financial support and wonderful research environment. We would like to thank Shuai Guo for contributing the appendix and many useful discussions. The first author would like to thank Naoki Koseki and Songtao Ma for useful comments. The second author would like to thank Albrecht Klemm and Rahul Pandharipande for valuable discussions.

\subsection*{Notation and conventions} \leavevmode

Throughout this paper, we work over the field of complex numbers $\mathbb{C}$.  We say a one-dimensional proper scheme $C$ is a \emph{curve} if it is Cohen--Macaulay. This is equivalent to saying that $C$ has no isolated or embedded points, or in other words, $\oh_C$ is a pure sheaf. 

We denote by $I_{Y/X}$ the ideal sheaf of a closed subscheme $Y$ in a scheme $X$. If $X$ is a given 3-fold, then we denote by $I_{Y}:=I_{Y/X}$ for simplicity.

For a line bundle $\mathcal{L}$, the linear series associated to $\mathcal{L}$ is denoted by $|\mathcal{L}|$. The Euler characteristic of a coherent sheaf $F$ on $X$ is denoted by $\chi(F):=\sum_{i} (-1)^i \dim\mathrm{H}^i(X,F)$.

Let $X$ be a smooth projective variety. If $X$ has Picard rank one, we denote by $\oh_X(1)$ the ample generator of $\Pic(X)$. The corresponding divisor will be denoted by $H:=c_1(\oh_X(1))$. If $C\subset X$ is a one-dimensional closed subscheme, we denote the (arithmetic) genus and degree of $C$ by $g(C)$ and $d(C):=C\cdot H$, respectively. For a 1-cycle $\beta$ on $X$, we define $\deg_H \beta:=H\cdot \beta$.

If $X$ is a smooth projective 3-fold of Picard rank one, we define the degree and index of $X$ by $n:=H^3$ and $K_X=-iH$, respectively.

We denote the subset of $H_2(X, \mathbb{Z})$ consists of curve classes by $H_2(X, \mathbb{Z})_{c}$. For example, if $X\subset \PP^4$ is a smooth hypersurface, then we have $H_2(X, \mathbb{Z})_{c}=\mathbb{Z}_{\geq 1}\cdot L$, where $L$ is the class of lines on $X$. 

For an object $E\in \D^b(X)$, we denote the $i$-th Chern character of $E$ by $\ch_i(E)$. For an integer $n\geq 0$, the $n$-truncated Chern character is defined by
    \[\ch_{\leq n}(E):=\big(\ch_0(E), \ch_1(E),\dots, \ch_n(E)\big).\]
For a real number $b\in \mathbb{R}$ and a divisor $H$, the $b$-twisted Chern character is
    \[\ch^{b H}(E):=\exp(-b H)\cdot \ch(E).\]
We will always take $H$ to be a fixed ample divisor and write $\ch^{b}_i(E):=\ch^{bH}_i(E)$ for simplicity. For any integer $n\geq 0$, we define 
\[\ch_{H, \leq n}(E):=\big(\ch_{H,0}(E), \ch_{H,1}(E),\dots, \ch_{H,n}(E)\big)\in \QQ^{n+1},\]
where
\[\ch_{H, i}(E):=\frac{H^{\dim X-i}\ch_i(E)}{H^{\dim X}}\in \mathbb{Q}.\]

\section{Stability conditions} \label{sec:stability_condition}

Let $X$ be a smooth projective variety and $\D^b(X)$ be the bounded derived category of coherent sheaves on $X$.
In this section, we recall the construction of (weak) Bridgeland stability conditions on $\D^b(X)$, and the notion of tilt-stability introduced in \cite{bridgeland, bayer2011bridgeland}. 

\subsection{Weak stability conditions}

We denote by $\KK(X)$ the K-group of $\D^b(X)$. We fix a surjective morphism $v \colon \KK(X) \twoheadrightarrow \Lambda$ to a finite rank lattice. 


\begin{definition}
Let $\cA$ be an abelian category and $Z \colon \KK(\cA) \ra \CC$ be a group homomorphism. We call $Z$ a \emph{weak stability function} on $\cA$ if for any $E \in \cA$ we have $\Im Z(E) \geq 0$ and if $\Im Z(E) = 0$ then $\Re Z(E) \leq 0$. If furthermore for $0\neq E\in \cA$ we have $\Im Z(E) = 0$ implies that $\Re Z(E) < 0$, then we call $Z$ a \emph{stability function} on $\cA$.
\end{definition}

\begin{definition}
A \emph{weak stability condition} on $\D^b(X)$ is a pair $\sigma = (\cA, Z)$ where $\cA$ is the heart of a bounded t-structure on $\D^b(X)$, and $Z \colon \Lambda \ra \CC$ is a group homomorphism such that 
\begin{enumerate}[(i)]
    \item the composition $Z \circ v \colon \KK(\cA) \cong \KK(X) \ra \CC$ is a weak stability function on $\cA$. From now on, we write $Z(E)$ rather than $Z(v(E))$.
\end{enumerate}
For any $E \in \cA$, we define the \emph{slope} of $E$ with respect to $\sigma$ as
\[
\mu_\sigma(E) := \begin{cases}  - \frac{\Re Z(E)}{\Im Z(E)}, & \text{if} ~ \Im Z(E) > 0 \\
+ \infty , & \text{otherwise}.
\end{cases}
\]
We say an object $0 \neq E \in \cA$ is $\sigma$-(semi)stable if $\mu_\sigma(F) < \mu_\sigma(E/F)$ (respectively, $\mu_\sigma(F) \leq \mu_\sigma(E/F)$) for all proper subobjects $F \sst E$. 
\begin{enumerate}[(i), resume]
    \item Any object $E \in \cA$ has a Harder--Narasimhan filtration in terms of $\sigma$-semistability defined above.
    \item There exists a quadratic form $Q$ on $\Lambda \otimes_{\ZZ} \mathbb{R}$ such that $Q|_{\ker Z}$ is negative definite, and $Q(E) \geq 0$ for all $\sigma$-semistable objects $E \in \cA$. This is known as the \emph{support property}.
\end{enumerate}
If the composition $Z \circ v$ is a stability function, then we say $\sigma$ is a \emph{stability condition} on $\D^b(X)$.
\end{definition}


\subsection{Tilt-stability}
Let $(X, H)$ be an $n$-dimensional polarised smooth projective variety. Starting with the classical slope stability, where  
\[
\mu_H(E) := \begin{cases}  \frac{H^{n-1}\ch_1(E)}{H^n\ch_0(E)}, & \text{if} ~ \ch_0(E)\neq 0 \\
+ \infty , & \text{otherwise}.
\end{cases}
\]
for any $E\neq 0\in \Coh(X)$, we can form the once-tilted heart $\cA^{b}$ for any $b\in \mathbb{R}$ as follows.
For the slope we just defined, every sheaf $E$ has a Harder-Narasimhan filtration. Its graded pieces have slopes whose maximum we denote by $\mu^+_H(E)$ and minimum by $\mu^-_H(E)$. Then for any $b\in \mathbb{R}$, we define an abelian category $\cA^b\subset \D^b(X)$ as
\[\cA^b:=\{E\in \D^b(X) \mid \mu^+_H\big(\cH^{-1}(E)\big)\leq b, b<\mu^-_H\big(\cH^0(E)\big),~\text{and}~\cH^i(E)=0~\text{for}~i\neq 0, -1\}.\]

It is a result of \cite{happel1996tilting} that the abelian category $\cA^b$ is the heart of a bounded t-structure on $\D^b(X)$ for any $b\in \mathbb{R}$.

Now for $E \in \cA^{b}$, we define
\[ Z_{a, b}(E) := \frac{1}{2} a^2 H^n \ch_0^{b H}(E) - H^{n-2} \ch_2^{b H}(E) + \mathfrak{i} H^{n-1} \ch_1^{b H}(E). \]

\begin{proposition}[{\cite{bayer2011bridgeland, bayer2016space}}] \label{tilt_stab}
Let $a>0$ and $b \in \mathbb{R}$. Then the pair $\sigma_{a, b} = (\cA^{b} , Z_{a, b})$ defines a weak stability condition on $\D^b(X)$. The quadratic form $Q$ is given by the discriminant
\begin{equation} \label{BG}
    \Delta_H(E) = \big(H^{n-1} \ch_1(E)\big)^2 - 2 H^n \ch_0(E) H^{n-2} \ch_2(E).
\end{equation}
The weak stability conditions $\sigma_{a, b}$ vary continuously as $(a, b) \in \mathbb{R}_{>0} \times \mathbb{R}$ varies. 
\end{proposition}

The weak stability conditions defined above are called \emph{tilt-stability conditions}. We denote the slope function of $\sigma_{a, b}$ by $\mu_{a, b}(-)$.


\begin{lemma} [{\cite[Lemma 2.7]{bayer2016space}, \cite[Proposition 4.8]{bayer2020desingularization}}] \label{bms lemma 2.7}
Let $E \in \D^b(X)$ be an object and $\mu_H(E)>b$. Then $E\in \cA^{b}$ and $E$ is $\sigma_{a, b}$-(semi)stable for $a \gg 0$ if and only if $E$ is a 2-Gieseker-(semi)stable sheaf.


\end{lemma}



\subsection{The generalised Bogomolov--Gieseker inequality} 

There is another Bogomolov--Gieseker-type inequality, which is conjectured by \cite{bayer2011bridgeland, bayer2016space} in order to construct stability conditions on $\D^b(X)$.

\begin{conjecture} [{\cite[Conjecture 4.1]{bayer2016space}}] \label{SBG_conj}
Let $(X, H)$ be a polarised smooth projective 3-fold. Assume that $E$ is any $\sigma_{a, b}$-semistable object. Then
\[Q_{a, b}(E):=a^2 \Delta_H(E)+4\big(H \ch_2^{bH}(E)\big)^2-6\big(H^2\ch_1^{bH}(E)\big) \ch_3^{bH}(E)\geq 0.\]
\end{conjecture}

This conjecture was proved in several cases. We list some results proved in \cite{MP1,MP2,Mac14,sch13,liquintic,li2018stability} which will be mentioned in later sections. We refer to \cite[Section 4]{BM22} for a more complete list of known results.

\begin{theorem} \label{SBG}
The Conjecture \ref{SBG_conj} holds when

\begin{enumerate}[(i)]
\item $X$ is a smooth Fano 3-fold of Picard number one (including three-dimensional hypersurfaces of degree $\leq 4$) and $(a,b)\in \mathbb{R}_{>0}\times \mathbb{R}$,

\item $X$ is a principally polarized abelian 3-fold of Picard rank one and $(a,b)\in \mathbb{R}_{>0}\times \mathbb{R}$, or

\item $X$ is a quintic 3-fold and $a^2>(b-\lfloor b\rfloor)(\lfloor b\rfloor+1-b)$.
    
\end{enumerate}

\end{theorem}


Note that our parameter $a$ is different from $\alpha$ used in \cite{liquintic} (c.f.~\cite[Remark 2.4]{liquintic}).

\begin{remark} \label{rmk_boundary}
If an object $E\in \D^b(X)$ satisfies $Q_{a,b}(E)\geq 0$ for any $(a,b)\in U$, where $U\subset \mathbb{R}_{>0}\times \mathbb{R}$ is a given subset, then by continuity, we have $Q_{a,b}(E)\geq 0$ for any $(a,b)\in \overline{U}$.
\end{remark}


\subsection{Wall-chamber structure} \label{wall_sec2}

Let $(X, H)$ be a polarised smooth projective 3-fold. In this subsection, we describe the wall-chamber structure of tilt-stability $\sigma_{a, b}$.

\begin{definition}
Let $v\in \KK(X)$. A \emph{numerical wall} for $v$ induced by $w\in \KK(X)$ is the subset
\[W(v, w)=\{(a, b)\in \mathbb{R}_{>0}\times \mathbb{R} \mid \mu_{a, b}(v)=\mu_{a, b}(w)\}\neq \varnothing.\]

Let $E\in \D^b(X)$ be an non-zero object with $[E]=v\in \mathrm{K}(X)$. A numerical wall $W=W(v,w)$ for $v$ is called an \emph{actual wall for $E$} if there is a short exact sequence of $\sigma_{a,b}$-semistable objects
\begin{equation} \label{wall_seq}
    0\to F\to E\to G\to 0
\end{equation}
in $\cA^{b}$ for one $(a,b)\in W$ (hence all $(a,b)\in W$ by \cite[Corollary  4.13]{bayer2020desingularization}) such that $W=W(E, F)$ and $[F]=w$.

We say such an exact sequence  \eqref{wall_seq} of $\sigma_{a,b}$-semistable objects  \emph{induces the actual wall $W$} for $E$.


A \emph{chamber} for $E$ is defined to be a connected component of the complement of the union of actual walls for $E$.

We say a point $(a,b)\in \mathbb{R}_{>0}\times \mathbb{R}$ is \emph{over} a numerical wall $W$ if $a>a'$ for any $(a',b)\in W$.
\end{definition}



\begin{proposition} [{\cite[Section 4]{bayer2020desingularization}}] \label{wall_prop}
Let $v\in \KK(X)$ and $\Delta_H(v)\geq 0$. Let $E\in \D^b(X)$ be an object with $[E]=v$.

\begin{enumerate}
    \item If $\mathcal{C}$ is a chamber for an object $E$, then $E$ is $\sigma_{a,b}$-semistable for some $(a,b)\in \cC$ if and only if it is for all $(a,b)\in \cC$.
    
    \item A numerical wall for $v$ is either a semicircle centered along the $b$-axis or a vertical wall parallel to the $a$-axis. The set of numerical walls for $v$ is locally finite. No two walls intersect.
    
    \item If $\ch_0(v)\neq 0$, then there is a unique numerical vertical wall $b=\mu_H(v)$. The remaining numerical walls are split into two sets of nested semicircles whose apex lies on the hyperbola $\mu_{a, b}(v)=0$.
    
    \item If $\ch_0(v)=0$ and $H^2\cdot \ch_1(v)\neq 0$, then every numerical wall is a semicircle whose apex lies on the ray $b=\frac{H\cdot \ch_2(v)}{H^2\cdot \ch_1(v)}$.
    
    \item If $\ch_0(v)=0$ and $H^2\cdot \ch_1(v)= 0$, then there are no actual walls for $v$.
    
    \item If an actual wall for $E$ is induced by a short exact sequence of $\sigma_{a,b}$-semistable objects
    \[0\to F\to E\to G\to 0\]
    then $\Delta_H(F)+\Delta_H(G)\leq \Delta_H(E)$ and equality can only occur if either $F$ or $G$ is a sheaf supported in dimension zero. 
    In particular, if $\mu_{a,b}(E)<+\infty$ then we have $\Delta_H(F)+\Delta_H(G)< \Delta_H(E)$
\end{enumerate}

\end{proposition}

By Proposition \ref{wall_prop}, a semicircle actual wall gives two adjacent chambers. Walls and chambers can be visualized as in \cite[Figure 1]{bayer2020desingularization}.

An easy consequence of (f) in Proposition \ref{wall_prop} is the following.

\begin{lemma} \label{delta_0}
Let $E$ be a $\sigma_{a_0,b_0}$-semistable object for some $(a_0,b_0)\in \mathbb{R}_{>0}\times \mathbb{R}$. Assume that $\Delta_H(E)=0$ and $\mu_{a_0,b_0}(E)<+\infty$, then $E$ is $\sigma_{a,b_0}$-semistable for any $a>0$.
\end{lemma}

\section{Walls} \label{wall_sec}
In this section, we are going to analyze the wall-chamber structure for ideal sheaves of curves.
Let $X$ be a smooth projective 3-fold such that $\Pic(X)$ is generated by an ample line bundle $\oh_X(1)$. A \emph{hypersurface section} of $X$ of degree $k\geq 1$ is an effective divisor in the linear series $|\oh_X(k)|$. 

Let $H:=c_1(\oh_X(1))$ and $n:=H^3$ be the degree of $X$. Then we have
\[\ch_{H,\leq 1}(-) \in \ZZ\times \ZZ.\]
We denote $\ch^{bH}_i(-)$ by $\ch^b_i(-)$. For any closed subscheme $Z\subset X$, its ideal sheaf in $X$ is denoted by $I_Z$. For any curve $C\subset X$ of degree $d$, we have $\ch_{H,\leq 2}(I_C)=(1,0,-\frac{d}{n})$. By Lemma \ref{bms lemma 2.7}, $I_C\in \cA^b$ is $\sigma_{a,b}$-stable for any $b<0$ and $a\gg 0$.

\begin{proposition} \label{wall}
Let $C\subset X$ be a curve with degree $d$.

\begin{enumerate}[(i)]
    \item There is no actual wall on $b=-1$ for $I_C$.
    
    \item Assume that for some $b<0$, there is an actual wall for $I_C$ induced by an exact sequence in $\cA^{b}$
\[0\to A\to I_C\to B\to 0.\]
Then $A$ is a torsion-free sheaf. 

\item Moreover, if $\rk(A)=1$, then $b<-1$ and $B$ is a torsion sheaf. In this case,  $A=\oh_X(-D)\cong \oh_X(-kH)$ or $A=I_{C_1}(-D)\cong I_{C_1}(-kH)$ for a hypersurface section $D\in |\oh_X(k)|$ of $X$, where $\lceil b \rceil \leq -k\leq -1$ and $C_1\subset C$ is a curve with degree
\[1\leq d_1:=d(C_1)< \min\{d-\frac{k^2 n}{2}, d+\frac{k^2 n}{2}-k\sqrt{2nd}\}.\]
\end{enumerate}

\end{proposition}

\begin{proof} 
The statement (i) follows from \cite[Lemma 2.7]{MS20}.
    
For (ii), we take the cohomology with respect to the heart $\Coh(X)$. By the definition of $\cA^b$, we have $\cH^0(A)\cong A$, $\cH^{i}(B)=0$ if $i\notin \{-1, 0\}$ and a long exact sequence in $\Coh(X)$
\[0\to \cH^{-1}(B)\to A\xra{\pi_1} I_C\to \cH^0(B)\to 0.\]
Since $\cH^{-1}(B)$ and $I_C$ are torsion-free, we know that $A$ is also torsion-free. 

Now we prove (iii). We assume that $\rk(A)=1$. Since the map $\pi_1\colon A\to I_C$ is non-trivial, it is injective in $\Coh(X)$. Thus $\cH^{-1}(B)=0$ and $B\cong \cH^0(B)\in \Coh(X)$ is a torsion sheaf.

By the definition of $\sigma_{a,b}$, the imaginary parts of central charges satisfy 
\[\Im\big(Z_{a,b}(A)\big)=H^2\ch^b_1(A)\geq 0, \quad \Im\big(Z_{a,b}(B)\big)=H^2\ch^b_1(B)\geq 0.\]
Thus $0\leq H^2\ch^b_1(A)\leq H^2\ch^b_1(I_C)$. But from $b\neq 0$, we obtain $\mu_{a,b}(A)=\mu_{a,b}(B)=\mu_{a,b}(I_C)<+\infty$ in this case. Therefore, we get
\[0< H^2\ch^b_1(A)< H^2\ch^b_1(I_C),\]
which implies
\begin{equation} \label{ch1A}
    bH^3<H^2\ch_1(A)<0
\end{equation}
Since $X$ has Picard rank one, we see $\ch_1(A)=-kH$ for an integer $k$. Then \eqref{ch1A} implies $b<-1$ and $\lceil b \rceil \leq -k\leq -1$.

If $A$ is reflexive, from the assumption that $A$ has rank one and  \cite[Proposition 1.9]{har80}, we see $A$ is a line bundle. Hence $A\cong \oh_X(-kH)$. If $A$ is not reflexive, then $A\cong I_{C_1}(-kH)$, where $C_1$ is a closed subscheme of $X$ with dimension $\leq 1$. Now by taking the double dual, we have a commutative diagram with exact rows:
\[\begin{tikzcd}
	0 & {I_{C_1}(-kH)} & {\oh_X(-kH)} & {\oh_{C_1}(-kH)} & 0 \\
	0 & {I_C} & {\oh_X} & {\oh_C} & 0.
	\arrow[from=1-1, to=1-2]
	\arrow[from=1-2, to=1-3]
	\arrow[from=1-3, to=1-4]
	\arrow[from=1-4, to=1-5]
	\arrow[from=2-4, to=2-5]
	\arrow[from=2-2, to=2-3]
	\arrow[from=2-3, to=2-4]
	\arrow[from=2-1, to=2-2]
	\arrow["{\pi_1}", from=1-2, to=2-2]
	\arrow[from=1-3, to=2-3]
	\arrow["{\pi_2}", from=1-4, to=2-4]
\end{tikzcd}\]
Using the snake lemma, we obtain an injection $\ker(\pi_2)\hookrightarrow \cok(\pi_1)=B$ in $\Coh(X)$. Since  $\ker(\pi_2)\in \cA^{b}$ and $\mu_{a, b}(\ker(\pi_2))=+\infty$, by the tilt-stability of $B$ we get $\ker(\pi_2)=0$. This implies that $\oh_{C_1}(-kH)$ is a subsheaf of  $\oh_C$. Hence $\oh_{C_1}$ is a pure sheaf and $C_1\subset C$ is a curve since $C$ is. Note that up to now, we see $d_1=d(C_1)\leq d$.

We have $\ch_{H, \leq 2}(I_{C_1})=\big(1,0,-\frac{d_1}{n}\big)$. Then 
\[\ch_{H, \leq 2}(A)=\big(1,-k,\frac{k^2}{2}-\frac{d_1}{n}\big),\quad \ch_{H, \leq 2}(B)=\big(0,k,\frac{d_1-d}{n}-\frac{k^2}{2}\big).\]

From the definition of walls, we obtain $\mu_{a,b}(A)=\mu_{a,b}(I_C)$ for some $(a,b)\in \mathbb{R}_{>0}\times \mathbb{R}$. Thus we see
\[\frac{\frac{k^2}{2}-\frac{d_1}{n}+kb+\frac{1}{2}(b^2-a^2)}{-k-b}=\frac{-\frac{d}{n}+\frac{1}{2}(b^2-a^2)}{-b},\]
which gives
\begin{equation} \label{circle}
    a^2+\big( b+(\frac{d-d_1}{kn}+\frac{k}{2}) \big)^2=\big(\frac{d-d_1}{kn}+\frac{k}{2}\big)^2-\frac{2d}{n}.
\end{equation}
Since \eqref{circle} has solutions $(a,b)$ in $\mathbb{R}_{>0}\times \mathbb{R}$, we see  $\big(\frac{d-d_1}{kn}+\frac{k}{2}\big)^2-\frac{2d}{n}>0$. Then by $d_1\leq d$, we conclude
\[d_1<d+\frac{k^2n}{2}-k\sqrt{2nd}.\]

Finally, from (f) in Proposition \ref{wall_prop} and since $A, B$ are not supported on points, we have $\Delta_H(A)+\Delta_H(B)< \Delta_H(I_C)$. Then we get
\[d_1< d-\frac{k^2n}{2}\]
and the result follows.

\end{proof}

Let $d$ be a positive integer, we define \[\rho_d:=\sqrt{\frac{d}{4n}},\quad b_d:=\rho_d-\sqrt{\rho_d^2+\frac{2d}{n}}=-\sqrt{\frac{d}{n}}.\] 

\begin{lemma} \label{rank_1}
Let $E\in \cA^{b}$ be a $\sigma_{a,b}$-semistable object for some $a>0$ with $\ch_{H,\leq 2}(E)=(1,0,-\frac{d}{n})$ and $d\in \mathbb{Z}_{>0}$. Assume that $b_d\leq b<0$ and there is an actual wall for $E$ induced by a subobject $F\hookrightarrow E$ or quotient $E\twoheadrightarrow F$ in $\cA^b$ with $\ch_0(F)\geq 1$. Then $\ch_0(F)=1$.
\end{lemma}

\begin{proof}
Assume that the actual wall $W(E, F)$ is given by 
\[W(E, F)=(b-b_0)^2+a^2=\rho_0^2.\]
Then by Proposition \ref{wall_prop} (c), we see $b_0=-\sqrt{\rho_0^2+\frac{2d}{n}}$. By assumption, we have \[W(E, F)\cap \{(a,b) \mid 0<a, ~b_d\leq b<0\}\neq \varnothing.\]
Thus one can verify \[\rho_0-\sqrt{\rho_0^2+\frac{2d}{n}}>b_d=\rho_d-\sqrt{\rho_d^2+\frac{2d}{n}},\]
which implies $\rho_0>\rho_d$. If $\ch_0(F)\geq 2$, since
\[\rho^2_0>\rho_d^2=\frac{\Delta_H(E)}{4(2H^3)(2H^3-H^3\ch_0(E))}\geq \frac{\Delta_H(E)}{4(H^3\ch_0(F))(H^3\ch_0(F)-H^3\ch_0(E))},\]
by \cite[Proposition 4.16]{bayer2020desingularization} we get $\ch_0(F)\leq \ch_0(E)=1$, which makes a contradiction. Then the result follows.
\end{proof}

As a corollary, we can control the type of walls when $b$ is not so negative.

\begin{corollary} \label{cor_wall}
Let $C\subset X$ be a curve with degree $d$. Assume that for some $b_d\leq b<0$ there is an actual wall for $I_C$ given by an exact sequence in $\cA^{b}$
\[0\to A\to I_C\to B\to 0.\]
Then $b_d\leq b<-1$ and $A= \oh_X(-D)$ or $A=I_{C_1}(-D)$ for a hypersurface section $D\in |\oh_X(k)|$ of $X$, where $\lceil b \rceil \leq -k\leq -1$ and $C_1\subset C$ is a curve with degree
\[d_1:=d(C_1)< \min\{d-\frac{k^2 n}{2}, d+\frac{k^2 n}{2}-k\sqrt{2nd}\}.\]

Moreover, in the case $A=I_{C_1}(-D)$, there is another one-dimensional closed subscheme $C_2\subset C\cap D$ of degree $d_2$ such that $d_1+d_2=d$ and 
\begin{equation} \label{key_equation}
    g(C)=g(C_1)+g(C_2)+kd_1-1.
\end{equation}
\end{corollary}

\begin{proof}
By Proposition \ref{wall}, we know that $A$ is a torsion-free sheaf, thus $\ch_0(A)\geq 1=\ch_0(I_C)$. Then from Lemma \ref{rank_1} and $b_d\leq b$, we see $\ch_0(A)=1$. Therefore, the first part of the statement can be deduced from (iii) of Proposition \ref{wall}.

To prove the second one, we have a commutative diagram as in Proposition \ref{wall} with exact rows:
\[\begin{tikzcd}
	0 & {I_{C_1}(-D)} & {\oh_X(-D)} & {\oh_{C_1}(-D)} & 0 \\
	0 & {I_C} & {\oh_X} & {\oh_C} & 0.
	\arrow[from=1-1, to=1-2]
	\arrow[from=1-2, to=1-3]
	\arrow[from=1-3, to=1-4]
	\arrow[from=1-4, to=1-5]
	\arrow[from=2-4, to=2-5]
	\arrow[from=2-2, to=2-3]
	\arrow[from=2-3, to=2-4]
	\arrow[from=2-1, to=2-2]
	\arrow["{\pi_1}", from=1-2, to=2-2]
	\arrow[from=1-3, to=2-3]
	\arrow["{\pi_2}", from=1-4, to=2-4]
\end{tikzcd}\]
Moreover, we have seen in Proposition \ref{wall} that all vertical arrows in the diagram are injective in $\mathrm{Coh}(X)$. Thus by the snake lemma, we can fill the above diagram as
\[\begin{tikzcd}
	& 0 & 0 & 0 \\
	0 & {I_{C_1}(-D)} & {\oh_X(-D)} & {\oh_{C_1}(-D)} & 0 \\
	0 & {I_C} & {\oh_X} & {\oh_C} & 0 \\
	0 & T & {\oh_D} & {\oh_{C_2}} & 0 \\
	& 0 & 0 & 0
	\arrow[from=2-1, to=2-2]
	\arrow[from=2-3, to=2-4]
	\arrow[from=2-4, to=2-5]
	\arrow[from=3-4, to=3-5]
	\arrow[from=3-2, to=3-3]
	\arrow[from=3-3, to=3-4]
	\arrow[from=3-1, to=3-2]
	\arrow["{\pi_1}", from=2-2, to=3-2]
	\arrow[from=2-3, to=3-3]
	\arrow["{\pi_2}", from=2-4, to=3-4]
	\arrow[from=2-2, to=2-3]
	\arrow[from=4-1, to=4-2]
	\arrow[from=4-2, to=4-3]
	\arrow[from=4-3, to=4-4]
	\arrow[from=4-4, to=4-5]
	\arrow[from=3-2, to=4-2]
	\arrow[from=3-3, to=4-3]
	\arrow[from=3-4, to=4-4]
	\arrow[from=4-4, to=5-4]
	\arrow[from=4-3, to=5-3]
	\arrow[from=4-2, to=5-2]
	\arrow[from=1-2, to=2-2]
	\arrow[from=1-3, to=2-3]
	\arrow[from=1-4, to=2-4]
\end{tikzcd}\]
with all rows and columns being exact. It is clear that $C_2$ defined in the above diagram satisfies $C_2\subset C\cap D$. And from the sequence
\[0\to \oh_{C_1}(-D)\xra{\pi_2} \oh_C\to \oh_{C_2}\to 0,\]
we obtain $d=d_1+d_2$ and   $\chi(\oh_C)=\chi(\oh_{C_1}(-D))+\chi(\oh_{C_2})$. Then the equation \eqref{key_equation}  follows from 
\[\chi(\oh_C)=\chi(\oh_{C_1}(-D))+\chi(\oh_{C_2})=\chi(\oh_{C_1})+\chi(\oh_{C_2})-kd_1.\]
\end{proof}

\section{Bound of the genus} \label{sec_bound}

In this section, we aim to prove Theorem \ref{BMT_thm}, \ref{bound_quintic} and Theorem \ref{20_51}.

\subsection{General cases}
In this subsection, we are going to prove Theorem \ref{BMT_thm}.
We fix $X$ to be a smooth projective 3-fold such that $\Pic(X)$ is generated by an ample line bundle $\oh_X(1)$ as in the previous section. Recall that $H:=c_1(\oh_X(1))$ and $n:=H^3$. We define the index $i$ of $X$ by $K_X=-iH$. By the Hirzebruch-Riemann-Roch Theorem, we obtain \[g(C)=\ch_3(I_C)+1-\frac{i}{2}d\]
for any one-dimensional closed subscheme $C\subset X$ of degree $d$.

First, we introduce a special kind of curves.

\begin{definition} \label{def_good}
Let $C\subset X$ be a curve of degree $d$. For an integer $k\geq 1$, we say $C$ is \emph{k-neutral\footnote{The name neutral is motivated by \emph{neutrino} in physics, a kind of particle that has no electrical charge. They rarely react with other matter and can cross any wall in the world.} (in $X$)} if there is no actual wall for $I_C$ in the range $a>0$ and $\max\{b_d, -k-1\}<b<0$.
\end{definition}

\begin{remark}

By definition, if $C\subset X$ is k-neutral, then $I_{C}$ is $\sigma_{a,b}$-stable for any $a>0$ and  $\max\{b_d, -k-1\}<b<0$.
If $C\subset X$ is a curve with degree $d\leq n$, then $b_d\geq -1$. By Corollary \ref{cor_wall}, such $C$ is k-neutral for any $k\geq 1$.
\end{remark}

To prove Theorem \ref{BMT_thm}, we begin with some lemmas.

\begin{lemma} \label{cm_lemma}
Let $C\subset \mathbb{P}^m$ be a one-dimensional closed subscheme with degree $d$ and genus $g$. If $C$ is not a curve, then there is a closed subscheme $C'\subset C$ such that $C'$ is a curve with $d(C')=d$ and $g(C')>g$.
\end{lemma}

\begin{proof}
Note that for a one-dimensional Noetherian scheme, being Cohen-Macaulay is equivalent to having no isolated or embedded points. In other word, it is equivalent to $\oh_C$ being pure.

Since $C$ is not Cohen-Macaulay, $\oh_{C}$ contains finitely many sheaves supported on points, corresponding to embedded and isolated points of $C$. Let $T$ be the maximal zero-dimensional subsheaf. We take $C'$ to be $\oh_{C'}:=\oh_{C}/T$. Then it is clear that $C'$ is Cohen-Macaulay with $d(C')=d$ and $g(C')>g$.
\end{proof}

According to Lemma \ref{cm_lemma}, to get an upper bound for the genus of all one-dimensional closed subschemes, we only need to consider curves.

\begin{lemma} \label{rank_0}
Let $m\geq 1$ be an integer. Let $E\in \cA^{b}\subset \D^b(X)$ be a $\sigma_{a,b}$-semistable object for some $a>0$ with $\ch_{H,\leq 2}(E)=\big(0,m,-\frac{d}{n}-\frac{m^2}{2}\big)$. Assume that there is an actual wall for $E$ induced by a subobject $F\hookrightarrow E$ or quotient $E\twoheadrightarrow F$ in $\cA^b$ with $\ch_0(F)\geq 0$. 

\begin{enumerate}[(i)]
    \item If $-\frac{d}{nm}\leq b$, then $\ch_0(F)=0$.
    
    \item If $-\frac{d}{nm}-\frac{m}{4}\leq b<-\frac{d}{nm}$, then $0\leq \ch_0(F)\leq 1$.
\end{enumerate}

\end{lemma}

\begin{proof}
The result immediately follows from \cite[Proposition 4.16]{bayer2020desingularization} as in Lemma \ref{rank_1}.
\end{proof}

\begin{lemma} \label{torsion_no_wall}
Let $D\in |\oh_X(m)|$ be an integral hypersurface section. Let $C\subset D$ be a one-dimensional closed subscheme of degree $d$. Then there is no actual wall for $I_{C/D}$ in the range $0<a$ and $-\frac{d}{nm}\leq b$. Thus $I_{C/D}$ is $\sigma_{a,b}$-semistable for any $a>0$ and $-\frac{d}{nm}\leq b$.
\end{lemma}

\begin{proof}
Since $D$ is integral, we know that $I_{C/D}$ is a Gieseker-stable sheaf on $X$. Then $I_{C/D}$ is $\sigma_{a,b}$-semistable for $a\gg 0$ and $b\in \mathbb{R}$ by Lemma \ref{bms lemma 2.7}.
Assume that there is an actual wall in the range $0<a$ and $-\frac{d}{nm}\leq b$ induced by an exact sequence of $\sigma_{a,b}$-semistable objects in $\cA^{b}$
\[0\to A\to I_{C/D}\to B\to 0.\]
From (i) in Lemma \ref{rank_0}, we know that $\ch_0(A)=\ch_0(B)=0$. Moreover, by taking cohomology, we see $\cH^0(A)=A\in \Coh(X)$ and have a long exact sequence
\[0\to \cH^{-1}(B)\to A\to I_{C/D}\to \cH^0(B)\to 0.\]
Since $A$ is a torsion sheaf and $\cH^{-1}(B)$ is torsion-free, we see $\cH^{-1}(B)=0$. Thus $A,B\in \Coh(X)$ and $A$ is a subsheaf of  $I_{C/D}$. From the fact that $D$ is integral and $I_{C/D}$ is a pure sheaf on $X$, if $A\neq 0$ it satisfies $\ch_{H,1}(A)=\ch_{H,1}(I_{C/D})=m$. But then $\mu_{a,b}(B)=+\infty$ and contradicts $\mu_{a,b}(I_{C/D})=\mu_{a,b}(B)$. Hence there is no actual wall for $I_{C/D}$ in the range $0<a$ and $-\frac{d}{nm}\leq b$ and the result follows.
\end{proof}

In the rest of this subsection, we will focus on a specific class of $X$.

\begin{proposition} \label{prop_hyperplane_general}
Let $X$ be a smooth projective 3-fold of Picard rank one with degree $n$ and index $i$. Assume that Conjecture \ref{SBG_conj} holds for $X$ and any $(a,b)\in \mathbb{R}_{>0}\times \mathbb{R}$. Let $C\subset X$ be a one-dimensional closed subscheme with degree $d$ and genus $g$, such that $C\subset D$ for an integral hypersurface section $D\in |\oh_{X}(m)|$. Then we have
\[g\leq \frac{1}{2nm}d^2+\frac{m-i}{2}d+1.\]
\end{proposition}

\begin{proof}
Recall that
\[\ch_{H}(I_{C/D})=\big(0, m, -\frac{d}{n}-\frac{m^2}{2}, \frac{m^3}{6}+\frac{1}{n}(g(C)-1+\frac{i}{2}d)\big).\]
By Lemma \ref{torsion_no_wall}, we can apply Conjecture \ref{SBG_conj} to $I_{C/D}$ and $(a,b)=(0,-\frac{d}{mn})$. Then the result follows.
\end{proof}

\begin{lemma} \label{d_small}
Let $X$ be a smooth projective 3-fold of Picard rank one with degree $n$ and index $i$. Assume that Conjecture \ref{SBG_conj} holds for $X$ and any $(a,-1)\in \mathbb{R}_{>0}\times \mathbb{R}$. Let $C\subset X$ be a one-dimensional closed subscheme of degree $d$. Then we have
\[g(C)\leq \frac{2}{3n}d^2+(\frac{1}{3}-\frac{i}{2})d+1.\]
In particular, if $d\leq n$, then we have
\[g(C)\leq \frac{1}{2n}d^2+\frac{1-i}{2}d+1.\]
\end{lemma}

\begin{proof}
The first inequality follows from applying Conjecture \ref{SBG_conj} to $I_C$ and $(a,b)=(0,-1)$.
When $d\leq n$, one can check that
\[\frac{2}{3n}d^2+(\frac{1}{3}-\frac{i}{2})d+1\leq \frac{1}{2n}d^2+\frac{1-i}{2}d+1.\]
\end{proof}

\begin{theorem} \label{BMT_thm}
Let $X$ be a smooth projective 3-fold of Picard rank one of degree $n$ and index $i$. Assume that $X$ satisfies Conjecture \ref{SBG_conj} for all $(a,b)\in \mathbb{R}_{>0}\times \mathbb{R}$. Then for any one-dimensional closed subscheme $C\subset X$ of degree $d$ and genus $g$, we have
\[g\leq \frac{1}{2n}d^2+\frac{1-i}{2}d+1.\]
\end{theorem}

\begin{proof}

By Lemma \ref{cm_lemma}, we can assume that $C$ is a curve.

\textbf{Case 1.} $C$ is contained in a hyperplane section $D$ of $X$.

Since $X$ has Picard rank one, we know that $D$ is integral. Then the result follows from the case $m=1$ in Proposition \ref{prop_hyperplane_general}.

\textbf{Case 2.} $C$ is 1-neutral.

In this case we know that $I_{C}$ is $\sigma_{a,b}$-stable for any $a>0$ and  $\max\{b_d, -2\}<b<0$.
If $d\geq \frac{16}{9}n$, then we have $b_d\leq -\frac{4}{3}$. In this case, the result follows from applying Conjecture \ref{SBG_conj} to $I_C$ and $(a,b)=(\sqrt{\frac{2}{9}}, -\frac{4}{3})$.

Now we assume that $d<\frac{16}{9}n$. Then $-2<b_d$. We apply Conjecture \ref{SBG_conj} to $I_C$ and $(a,b)=(0, b_d)$ and get
\[g\leq \frac{2}{-3nb_d}d^2+\frac{-b_d}{3}d+1-\frac{i}{2}d.\]
Thus we only need to show

\begin{equation}\label{eq4}
\frac{2}{-3nb_d}d+\frac{-b_d}{3}\leq \frac{1}{2n}d+\frac{1}{2}.
\end{equation}

Since $b_d=-\sqrt{\frac{d}{n}}$, \eqref{eq4} is equivalent to
\[-b_d\leq \frac{1}{2n}d+\frac{1}{2}=\frac{b_d^2+1}{2},\]
which of course always holds.

\textbf{Case 3.} $C$ is not contained in any hyperplane section of $X$ and is not 1-neutral.

We prove this case by induction on $d(C)$.
If $d(C)\leq n$, the result follows from Lemma \ref{d_small}. Now let $d\geq n+1$ be an integer such that
\[g(C)\leq \frac{1}{2n}d(C)^2+\frac{1-i}{2}d(C)+1\]
holds for any curve $C\subset X$ of degree $d(C)<d$ that is not 1-neutral and not contained in any hyperplane section. We are going to deal with the case $d(C)=d$. 
By definition and Corollary \ref{cor_wall}, we have
\[g(C)=g(C_1)+g(C_2)+d_1-1,\]
where $C_1\subset C$ is a curve of degree $d_1$ and $C_2\subset C$ is a one-dimensional closed subscheme of degree $d_2$ satisfying $d_1+d_2=d$ and
\[d-1\geq d_2> \max\{\frac{n}{2}, \sqrt{2nd}-\frac{n}{2}\}.\]
Since $d_i<d$, by the induction hypothesis,  Lemma \ref{cm_lemma} and previous two cases, we see
\[g(C_i)\leq \frac{1}{2n}d_i^2+\frac{1-i}{2}d_i+1\]
holds for each $i$.
Then we obtain
\[g(C)=g(C_1)+g(C_2)+d_1-1\leq \frac{1}{2n}d^2+\frac{1-i}{2}d+1+d_1(1-\frac{d_2}{n}).\]
Moreover, by (b) and (c) of Proposition \ref{wall_prop}, we can assume that $W:=W(I_{C_1}(-H), I_C)$ is the upper-most actual wall for $I_C$, i.e.~ $I_C$ is $\sigma_{a,b}$-semistable for any $(a,b)$ over the wall $W$. If $d_2\geq n$, then we are done. Now assume that $d_2\leq n-1$. Then by Proposition \ref{wall}, the actual wall given by $I_{C_1}(-H)$ is
\[W:=W(I_{C_1}(-H), I_C): a^2+\big(b+(\frac{d_2}{n}+\frac{1}{2})\big)^2=(\frac{d_2}{n}+\frac{1}{2})^2-\frac{2d}{n}.\]
Since $d_2\leq n-1$, we can apply Conjecture \ref{SBG_conj} to $I_C$ and
\[(a,b)=\big(0,-\frac{3}{2}+\frac{1}{n}+\sqrt{(\frac{3}{2}-\frac{1}{n})^2-\frac{2d}{n}}\big),\]
which gives
\[g(C)\leq (1-\frac{2}{3n})d-\frac{i}{2}d+1.\]
Since $d\geq n+1>n-\frac{4}{3}$, we obtain
\[g(C)\leq (1-\frac{2}{3n})d-\frac{i}{2}d+1<\frac{1}{2n}d^2+\frac{1-i}{2}d+1\]
and finish the proof.

\end{proof}

\subsection{Three-dimensional hypersurfaces} \label{sec_4.1}

In this subsection, we are going to get a more refined bound for the genus of curves in a smooth hypersurface $X\subset \PP^4$ of degree $n$. By the adjunction formula, the index of $X$ is $i=5-n$.

\begin{lemma} \label{snna}
Let $C\subset \mathbb{P}^m$ be a one-dimensional closed subscheme with degree $d$ and genus $g$.

\begin{enumerate}[(i)]
    \item If $d=1$, then $g(C)\leq 0$. And $g(C)=0$ if and only if $C$ is a line.
    
    \item If $d=2$, then $g(C)\leq 0$. And $g(C)=0$ if and only if $C$ is a plane conic.
    
    \item If $d=3$, then $g(C)\leq 1$. And $g(C)=1$ if and only if $C$ is a plane cubic.
\end{enumerate}
\end{lemma}

\begin{proof}
The proof is standard. See e.g.~\cite[Corollary 1.38]{sanna2014rational}.
\end{proof}

\begin{proposition} \label{prop_hyperplane}
Let $C\subset \PP^3$ be a one-dimensional closed subscheme with degree $d$ and genus $g$, such that $C\subset D$ for an integral surface $D\in |\oh_{\PP^3}(m)|$ of degree $m\geq 1$. Then we have
\[g\leq \frac{1}{2m}d^2+\frac{m-4}{2}d+1.\]
\end{proposition}

\begin{proof}
This follows from Theorem \ref{SBG} and Proposition \ref{prop_hyperplane_general}.
\end{proof}

\begin{remark}
This proposition can be regarded as a generalization of \cite{harris1980}.
\end{remark}

\begin{corollary} \label{in D}
Let $X\subset \PP^4$ be a smooth hypersurface of degree $n\geq 1$. Let $C\subset X$ be a one-dimensional closed subscheme of degree $d$ contained in some hyperplane sections of $X$. Then
\[g(C)\leq \frac{1}{2n}d^2+\frac{n-4}{2}d+1.\]
\end{corollary}

\begin{proof}
If $C$ is contained in some hyperplane sections $D\in |\oh_X(1)|$, then we can view $C\subset D$ as a curve in $\PP^3$. Since $X$ has Picard rank one, we know that $D\subset \PP^3$ is integral. Then the result follows from Proposition \ref{prop_hyperplane}.
\end{proof}

\begin{proposition} \label{lem_quintic}
Let $X\subset \PP^4$ be a smooth hypersurface of degree $n\leq 5$. Let $C\subset X$ be a one-dimensional closed subscheme of degree $d>n$ and not contained in any hyperplane section of $X$. Then
\[g(C)\leq \frac{1}{2n}d^2+(\frac{n}{2}-\frac{1}{n}-2)d+2+\frac{1}{n}\leq \frac{1}{2n}d^2+\frac{n-4}{2}d+1.\]
\end{proposition}

\begin{proof}
By Lemma \ref{cm_lemma}, we can assume that $C$ is a curve. Since $d\geq n+1$, the second inequality holds automatically. Then we only need to prove the first one.

If $d=3$, we have $1\leq n\leq 2$. By Lemma \ref{snna}, $g(C)\leq 1$ and $g(C)=1$ if and only if $C$ is a plane cubic. Since $X$ does not contain any plane, we know that $\PP^2\cap X$ is a curve of degree $n\leq 2$, which can not contain $C$. Then in this case, we obtain $g(C)\leq 0$. When $d=2$ and $n=1$, from Lemma \ref{snna} we get $g(C)\leq 0$. Then our statement holds in these cases. Therefore, in the following, we can assume that $d> \max\{n,3\}$.


\textbf{Case 1.} $n=5$ and $C$ is 1-neutral.

By the definition of 1-neutral, we can apply Theorem \ref{SBG} to $I_C$ and 
\[(a,b)=(\sqrt{(b-\lfloor b\rfloor)(\lfloor b\rfloor+1-b)},\max\{-2,b_d\}).\]
When $n+1\leq d\leq 19$, since $g(C)\in \ZZ$ and $n=5$, one can check case-by-case that we obtain \[g(C)\leq \frac{1}{2n}d^2+(\frac{n}{2}-\frac{1}{n}-2)d+2+\frac{1}{n} .\]
When $d\geq 20$, we have $\max\{-2,b_d\}=-2$. Thus applying Theorem \ref{SBG} to $I_C$ and $(a,b)=(0,-2)$ implies
\[g(C)\leq \frac{1}{15}d^2+\frac{2}{3}d+1\leq \frac{1}{2n}d^2+(\frac{n}{2}-\frac{1}{n}-2)d+2+\frac{1}{n}.\]

\textbf{Case 2.} $n\leq 4$ and $C$ is 1-neutral.

When $d\geq 4n$, we have $\max\{-2,b_d\}=-2$. Then as the first case, applying Theorem \ref{SBG} to $I_C$ and $(a,b)=(0,-2)$ gives
\[g(C)\leq \frac{1}{3n}d^2+\frac{3n-11}{6}d+1\leq \frac{1}{2n}d^2+(\frac{n}{2}-\frac{1}{n}-2)d+2+\frac{1}{n}.\]
When $\max\{n,3\}<d<4n$, we apply Theorem \ref{SBG} to $I_C$ and 
\[(a,b)=(0,\max\{-2,b_d\})\]
as in the first case. Since $g(C)\in \ZZ$, it is immediately to check case-by-case that
\[g(C)\leq \frac{1}{2n}d^2+(\frac{n}{2}-\frac{1}{n}-2)d+2+\frac{1}{n}.\]

\textbf{Case 3.} $C$ is not 1-neutral.

We prove this case by induction on $d$. First we assume that $d=n+1$. When $n\leq 2$, the result is proved at the beginning of this proof. When $3\leq n\leq 5$, the result follows from Lemma \ref{d_small}.

Now let $d'\geq n+2$ be an integer such that
\[g(C)\leq \frac{1}{2n}d^2+(\frac{n}{2}-\frac{1}{n}-2)d+2+\frac{1}{n}\]
holds for any curve $C\subset X$ of degree $d<d'$ that is not 1-neutral and not contained in any hyperplane section. We are going to deal with the case $d=d'$. 

By definition and Corollary \ref{cor_wall}, we have
\[g(C)=g(C_1)+g(C_2)+d_1-1,\]
where $C_1\subset C$ is a curve of degree $d_1$ and $C_2\subset C$ is a one-dimensional closed subscheme of degree $d_2$ satisfying $d_1+d_2=d$ and
\[d-1\geq d_2> \max\{\frac{n}{2}, \sqrt{2nd}-\frac{n}{2}\}=\sqrt{2nd}-\frac{n}{2}.\]
Since $d_i<d=d'$, by the induction hypothesis, Corollary \ref{in D}, Lemma \ref{d_small} and previous two cases, we see
\[g(C_i)\leq \frac{1}{2n}d_i^2+\frac{n-4}{2}d_i+1\]
for each $i$. Then we obtain
\[g(C)=g(C_1)+g(C_2)+d_1-1\leq \frac{1}{2n}d^2+\frac{n-4}{2}d+1+d_1(1-\frac{d_2}{n}).\]
Therefore, we only need to prove
\[d_1(1-\frac{d_2}{n})\leq 1+\frac{1-d}{n}.\]
To this end, note that
\[d_1(1-\frac{d_2}{n})-1-\frac{1-d}{n}=\frac{1}{n}(d_2-n-1)(d_2-d+1).\]
Since $d_2\leq d-1$, we only need to show $d_2\geq n+1$. When $d=n+2$, one can check case-by-case that $\lceil \sqrt{2nd}-\frac{n}{2}\rceil=n+1$. When $d\geq n+3$, since $1\leq n\leq 5$, we always have
\[\sqrt{2nd}-\frac{n}{2}>n+1.\]
Since $d_2\in \ZZ$ and $d_2>\sqrt{2nd}-\frac{n}{2}$, this finishes our proof.

\end{proof}


\begin{theorem} \label{bound_quintic}
Let $X\subset \PP^4$ be a smooth hypersurface of degree $n\leq 5$. Then for any one-dimensional closed subscheme $C\subset X$ of degree $d$ and genus $g$, we have
\[g\leq \frac{1}{2n}d^2+\frac{n-4}{2}d+1.\]
Moreover, when $C$ is not contained in any hyperplane section of $X$, we have
    \[g\leq \frac{1}{2n}d^2+(\frac{n}{2}-\frac{1}{n}-2)d+2+\frac{1}{n}.\]
\end{theorem}

\begin{proof}
When $d\leq n$, these two bound can be both deduced from Lemma \ref{d_small}. Thus we can assume that $d\geq n+1$.  Then the result follows from Corollary \ref{in D} and Proposition \ref{lem_quintic}.
\end{proof}

\subsection{Extremal curves} \label{extremal}

Let $X\subset \PP^4$ be a smooth hypersurface of degree $n\leq 5$.
By Proposition \ref{lem_quintic}, we can see that when $C$ has degree $d>n+1$ and is not contained in any hyperplane section, the genus $g$ of $C\subset X$ satisfies
\[g\leq \frac{1}{2n}d^2+(\frac{n}{2}-\frac{1}{n}-2)d+2+\frac{1}{n}<\frac{1}{2n}d^2+\frac{n-4}{2}d+1.\]
Thus for curves with genus $g=\frac{1}{2n}d^2+\frac{n-4}{2}d+1$, we have the following description.

\begin{lemma} \label{extremal_curve}
Let $X\subset \PP^4$ be a smooth hypersurface of degree $n\leq 5$. Let $C\subset X$ be a one-dimensional closed subscheme of degree $d\geq n+1$ and genus $g=\frac{1}{2n}d^2+\frac{n-4}{2}d+1$. Then $C$ is a curve contained in a unique hyperplane section of $X$.
\end{lemma}

\begin{proof}
By Lemma \ref{cm_lemma} and Theorem \ref{bound_quintic}, $C$ is a curve. If $d>n+1$, the above argument shows that $C$ is contained in some hyperplane sections of $X$.

When $n=1$ and $d=2$, the result follows from Lemma \ref{snna}. Thus when $n=1$, we can assume that $d>n+1$. From $n\leq 5$ and $g=\frac{1}{2n}d^2+\frac{n-4}{2}d+1\in \ZZ$, one can check that $d\in n\ZZ$. Thus if $2\leq n\leq 5$, we get $d\neq n+1$ and hence $d>n+1$. Then when $d\geq n+1$, we can conclude in all cases that $C$ is contained in some hyperplane sections of $X$.

Note that if $C$ is contained in two different hyperplane sections of $X$, then we have $d(C)\leq n$. Thus the uniqueness part follows.
\end{proof}

In the rest of this subsection, we work on $\PP^3$. All weak stability conditions and Chern characters are on $\PP^3$.

\begin{proposition}\label{prop_D}
Let $D\subset \PP^3$ be an integral surface of degree $n\geq 1$ and $C\subset D$ be a one-dimensional closed subscheme of degree $d$ and genus $g$. Then
\[g=\frac{1}{2n}d^2+\frac{n-4}{2}d+1\]
if and only if $d\in n\ZZ$ and $C\in |\oh_D(\frac{d}{n})|$.
\end{proposition}

\begin{proof}
By the adjunction formula, it is not hard to see that a curve $C\in |\oh_{D}(\frac{d}{n})|$ has the genus $\frac{1}{2n}d^2+\frac{n-4}{2}d+1$.

Conversely, we assume that $g=\frac{1}{2n}d^2+\frac{n-4}{2}d+1$.
Recall that
\[\ch_{H}(I_{C/D})=\big(0, n, -d-\frac{n^2}{2}, \frac{n^3}{6}+(g(C)+2d-1)\big).\]
Thus by Lemma \ref{torsion_no_wall}, there is no actual wall for $I_{C/D}$ on $b=-\frac{d}{n}$.
Note that by Proposition \ref{wall_prop}, if there is a numerical wall $W$ tangents to the line $b=-\frac{d}{n}$, it is given by
\[a^2+(b+\frac{d}{n}+\frac{n}{2})=\frac{n^2}{4}.\]
Then if $W$ is not an actual wall for $I_{C/D}$, by the local finiteness of walls, there is no actual wall for $I_{C/D}$ on $b=-\frac{d}{n}-\epsilon$ for $\epsilon>0$ sufficiently small. Applying Theorem \ref{SBG} to $I_{C/D}$ and $(a,b)=(0,-\frac{d}{n}-\epsilon)$, we obtain 
\[g(C)\leq \frac{2}{3n}d^2+\frac{b+2n-6}{3}d+\frac{b^2n+bn^2+6}{6}\]
\[<\frac{1}{2n}d^2+\frac{n-4}{2}d+1\] and get a contradiction. Therefore, $W$ is an actual wall.

Note that $(a,b)=(\frac{n}{2},-\frac{d}{n}-\frac{n}{2})\in W$, we can assume that the actual wall $W$ is given by an exact sequence
\[0\to A\to I_{C/D}\to B\to 0\]
in $\cA^{-\frac{d}{n}-\frac{n}{2}}$. Then by taking cohomology long exact sequence, we see $A\in \Coh(X)$ and thus $\ch_0(A)\geq 0$. By (ii) in Lemma \ref{rank_0}, we see $\ch_0(A)\leq 1$. Since $D$ is integral, the same argument as in Lemma \ref{torsion_no_wall} shows that $\ch_0(A)\neq 0$. Thus we have $\ch_0(A)=1$. 

Now we assume that
\[\ch_H(A)=\big(1,x,y,z\big), \quad \ch_H(B)=\big(-1,n-x,-d-\frac{n^2}{2}-y, \frac{n^3}{6}+(g(C)+2d-1)-z\big).\]
Using the definition and properties of walls in Section \ref{wall_sec2}, we see

\begin{enumerate}[(i)]
    \item $\mu_{\frac{n}{2},-\frac{d}{n}-\frac{n}{2}}(A)=\mu_{\frac{n}{2},-\frac{d}{n}-\frac{n}{2}}(I_{C/D})$ gives
    \[\frac{y+(\frac{d}{n}+\frac{n}{2})x+\frac{1}{2}((\frac{d}{n}+\frac{n}{2})^2-\frac{n^2}{4})}{x+\frac{d}{n}+\frac{n}{2}}=0.\]
    
    \item $\Delta_H(A)\geq 0$ gives  $x^2-2y\geq 0$.
    
    \item $\Delta_H(B)\geq 0$ gives $(n-x)^2-2(d+\frac{n^2}{2}+y)\geq 0$.
    
    \item $\Im\big(Z_{\frac{n}{2},-\frac{d}{n}-\frac{n}{2}}(A)\big)> 0$ and  $\Im\big(Z_{\frac{n}{2},-\frac{d}{n}-\frac{n}{2}}(B)\big)> 0$ gives $0<x+\frac{d}{n}+\frac{n}{2}<n$.
\end{enumerate}
From (i), we know
\begin{equation} \label{eq_1}
    y=-(\frac{d^2}{2n^2}+\frac{n+2x}{2n}d+\frac{nx}{2})=-\big((\frac{d}{n}+\frac{n}{2})x+\frac{d^2}{2n^2}+\frac{d}{2}\big).
\end{equation}
Then by (ii) and \eqref{eq_1}, we obtain
\begin{equation} \label{eq2}
    x^2+(\frac{2d}{n}+n)x+\frac{d^2}{n^2}+d=(x+\frac{d}{n})(x+\frac{d}{n}+n)\geq 0.
\end{equation}
From (iii) and \eqref{eq_1}, we get
\begin{equation} \label{eq3}
    x^2+(\frac{2d}{n}-n)x+\frac{d^2}{n^2}-d=(x+\frac{d}{n})(x+\frac{d}{n}-n)\geq 0.
\end{equation}
Now using (iv), \eqref{eq2} and \eqref{eq3}, we have
\[x=-\frac{d}{n},\quad y=\frac{d^2}{2n^2}.\]
Since $x\in \ZZ$, we know that $d\in n\ZZ$ and
\[\ch_{\leq 2}(A)=\ch_{\leq 2}\big(\oh_{\PP^3}(-\frac{d}{n})\big).\]

Now we are going to show $C\in |\oh_D(\frac{d}{n})|$. We only need to prove \[\Hom_{\D^b(\PP^3)}\big(\oh_{\PP^3}(-\frac{d}{n}\big), I_{C/D})=\mathrm{H}^0\big(D, I_{C/D}(\frac{d}{n})\big)\neq 0.\]
Since $\Delta_H(A)=0$, $A$ is $\sigma_{a, -\frac{d}{n}-\frac{n}{2}}$-semistable for any $a>0$ by Lemma \ref{delta_0}. Thus from Lemma \ref{bms lemma 2.7}, $A$ is a 2-Gieseker-semistable sheaf on $\PP^3$. Now applying Theorem \ref{SBG}, we know
\[Q_{\frac{n}{2},-\frac{d}{n}-\frac{n}{2}}(A)\geq 0, \quad Q_{\frac{n}{2},-\frac{d}{n}-\frac{n}{2}}(B)\geq 0.\]
From $Q_{\frac{n}{2},-\frac{d}{n}-\frac{n}{2}}(A)\geq 0$ we see $z\leq -\frac{d^3}{6n^3}$ and from $Q_{\frac{n}{2},-\frac{d}{n}-\frac{n}{2}}(B)\geq 0$ we see $z\geq -\frac{d^3}{6n^3}$.
Then we have $\ch_3(A)=\ch_3(\oh_{\PP^3}(-\frac{d}{n}))$. Thus $\ch(A)=\ch(\oh_{\PP^3}(-\frac{d}{n}))$ and we obtain $A\cong \oh_{\PP^3}(-\frac{d}{n})$. Therefore, we get
\[\Hom_{\D^b(\PP^3)}\big(A, I_{C/D}\big)=\Hom_{\D^b(\PP^3)}\big(\oh_{\PP^3}(-\frac{d}{n}), I_{C/D}\big)\neq 0\]
as desired and the result follows. 
\end{proof}

\begin{theorem} \label{20_51}
Let $X\subset \PP^4$ be a smooth hypersurface of degree $n\leq 5$. Let $C$ be a one-dimensional closed subscheme $C\subset X$ of degree $d\geq n+1$ and genus $g$. Then
\[g=\frac{1}{2n}d^2+\frac{n-4}{2}d+1\]
if and only if $C$ is a complete intersection of a hyperplane section and a degree $\frac{d}{n}$  hypersurface section of $X$.
\end{theorem}

\begin{proof}
By Lemma \ref{extremal_curve}, $C$ is a curve contained in a unique hyperplane section $D\subset X$. Since $X$ has Picard rank one, we know that $D\subset \PP^3$ is integral. Then the result follows from Proposition \ref{prop_D}.
\end{proof}

\begin{remark} \label{main_rmk}
Although our proof relies on Theorem \ref{SBG}, we only used a much weaker version of it. We hope to generalize Theorem \ref{bound_quintic} and \ref{20_51} to all complete intersection 3-folds in future research. A work regarding Castelnuovo-type bound and its improvement on complete intersection Calabi--Yau 3-folds is in preparation \cite{F2022}.
\end{remark}

\begin{remark}\label{char_p}
The results in Section \ref{sec_4.1} and \ref{extremal} remain valid without any change when we replace the base field $\mathbb{C}$ with an algebraically closed field $\mathbb{K}$ of \emph{any characteristic}, except those for quintic 3-folds. Indeed, our proof only relies on

\begin{itemize}
    \item the classical Bogomolov--Gieseker (BG) inequality \eqref{BG}, which implies Proposition \ref{tilt_stab}; and
    
    \item Bayer--Macr\'i--Toda's generalized BG inequality (Theorem \ref{SBG}).
\end{itemize}

For Fano or Calabi--Yau varieties over $\mathbb{K}$, the classical BG inequality was proved by \cite{lang04}. And the argument in \cite{li2018stability} also works for smooth hypersurfaces in $\PP^4_{\mathbb{K}}$ of degree $\leq 4$, which shows that Theorem \ref{SBG} holds in these cases. To prove the generalized BG inequality for quintic 3-folds over $\mathbb{K}$ via the same approach in \cite{liquintic}, we need the classical BG inequality for $(2,5)$-complete intersection surfaces over $\mathbb{K}$. As Naoki Koseki pointed out to us, this can probably be worked out by using the same techniques in \cite{koseki2020bogomolov}.
\end{remark}

\section{Vanishing of Gopakumar--Vafa invariants} \label{sec_vanish_GV}

The main goal of this section is to prove Theorem \ref{main_thm_1.1}.

\subsection{Pandharipande--Thomas invariants and  Gopakumar--Vafa invariants}

Let $(X,H)$ be a polarised smooth projective 3-fold. Recall that a two-term complex of coherent sheaves on $X$
\[[\oh_X\xra{s} F]\]
is called a \emph{stable pair} if $F$ is pure one-dimensional and the cokernel of $s$ is zero-dimensional.

Let $C$ be the scheme-theoretic support of $F$. Let $H_2(X, \mathbb{Z})_{c}$ be the subset of $H_2(X, \mathbb{Z})$ consists of curve classes. For any integer $n$ and curve class $\beta\in H_2(X, \mathbb{Z})_c$, let $P_n(X, \beta)$ be the moduli space of stable pairs $[\oh_X\xra{s} F]$
on $X$ with
\[\chi(F)=n,\quad[C]=\beta.\]
It is shown in \cite{PT07} that the moduli space $P_n(X, \beta)$ is a projective scheme with a virtual cycle $[P_n(X, \beta)]^{vir}$ of virtual dimension
\[\int_{\beta} c_1(X).\]
In particular, if $X$ is Calabi--Yau, the virtual cycle is zero-dimensional and we define a \emph{Pandharipande--Thomas (PT) invariant} by
\[P_{n,\beta}:=\int_{[P_n(X, \beta)]^{vir}} 1.\]


The generating series of PT-invariants is
\[\PT(q, t):=1+\sum_{\beta\neq 0}\sum_{n\in \ZZ} P_{n, \beta} q^n t^{\beta}.\]
Note that when $\beta$ is not a curve class, we have $P_{n, \beta}=0$ for any $n$.

\begin{corollary} \label{PT_zero}
Let $d\geq 1$ be an integer, and $B(-)\colon \mathbb{Z}_{>0}\to \mathbb{R}$ be a function such that $g(C)\leq B(d(C))$ for any curve $C\subset X$. Let $\beta\in H_2(X, \mathbb{Z})_{c}$ such that $\deg_H \beta=d$. Then the moduli space $P_n(X,\beta)=\varnothing$ for any $n<1-B(d)$.

In particular, if $X$ is Calabi--Yau, then we have
\[P_{n, \beta}=0\]
for any $n<1-B(d)$.
\end{corollary}

\begin{proof}
Let $[\oh_X\xra{s} F]\in P_n(X, \beta)$ be a stable pair. Then the image of $s$ is of the form $\oh_C$, where $C\subset X$ is the scheme-theoretic support of $F$. From the fact that $F$ is pure, $\oh_C$ is also pure. Hence $C$ is a curve.

Since the cokernel of $s$ is zero-dimensional, we have $\chi(\oh_C)\leq \chi(F)=n$. Then from $n<1-B(d)$, we see $g(C)>B(d)$ and such curve $C$ cannot exist by assumption. This means the moduli space of stable pairs $P_n(X, \beta)$ is empty for any $n<1-B(d)$. If $X$ is Calabi--Yau, by definition we get $P_{n, \beta}=0$ for any $n<1-B(d)$.
\end{proof}




\subsection{Vanishing}

Let $(X, H)$ be a polarised smooth Calabi--Yau 3-fold. 
Recall that for any non-zero effective 1-cycle $\beta\in H_2(X,\ZZ)$ and $g\geq 0$, the \emph{Gopakumar--Vafa (GV) invariant} $n_g^{\beta}$ is defined as the coefficient in an expansion of the series
\begin{equation} \label{GV_from_GW}
    \GW(\lambda, t)=\sum_{g\geq 0, \beta\neq 0} N_{g,\beta}\lambda^{2g-2}t^{\beta}=\sum_{g\geq 0,r\geq 1, \beta\neq 0}  \frac{n_g^{\beta}}{r}\cdot \big(2\sin(\frac{r\lambda}{2})\big)^{2g-2}\cdot t^{r\beta}.
\end{equation}

As in \cite[Section 3]{PT07}, we can write the connected PT series $\sF_{P}(q, t)$ uniquely as
\begin{equation} \label{GV_PT_1030}
    \sF_{P}(q, t):=\log \PT(q, t)=\sum_{g>-\infty} \sum_{\beta\neq 0} \sum_{r\geq 1} \overline{n}_g^{\beta} \frac{(-1)^{g-1}}{r} \big((-q)^{\frac{r}{2}}-(-q)^{-\frac{r}{2}}\big)^{2g-2} t^{r\beta},
\end{equation}
such that $\overline{n}_g^{\beta}=0$ for a fixed $\beta$ and $g\gg 0$.

Now we assume that the GW/PT correspondence \cite{PP12} holds for $X$ (e.g.~$X$ is a complete intersection Calabi--Yau 3-fold \cite[Corollary 2]{PP12}). After changing the variable $q=-\exp(i\lambda)$, the parts of series \eqref{GV_from_GW} and \eqref{GV_PT_1030} containing $t^{\beta}$ are equal \emph{as rational functions} in $q$ for any $\beta\neq 0$. Now by \cite{DIW21}, we know that $n^{\beta}_g=0$ for any fixed $\beta$ and $g\gg 0$. Thus from the uniqueness of the expansions \cite[Lemma 3.9, 3.11]{PT07}, we have
\[\overline{n}_g^{\beta}=n_g^{\beta},\]
or in other words,
\begin{equation} \label{GV_PT_5}
    \sF_{P}(q, t)=\log \PT(q, t)=\sum_{g\geq 0} \sum_{\beta\neq 0} \sum_{r\geq 1} n_g^{\beta} \frac{(-1)^{g-1}}{r} \big((-q)^{\frac{r}{2}}-(-q)^{-\frac{r}{2}}\big)^{2g-2} t^{r\beta}.
\end{equation}
Note that
\begin{equation}\label{expand}
    \big((-q)^{\frac{r}{2}}-(-q)^{-\frac{r}{2}}\big)^{2g-2}=\big( \frac{(-q)^r}{(1-(-q)^r)^2} \big)^{1-g} =(-q)^{r(1-g)}+2(1-g)(-q)^{r(2-g)}+O(q^{r(2-g)}).
\end{equation}

Using PT-invariants and formula \eqref{GV_PT_5}, we can prove our vanishing result of GV-invariants Theorem \ref{main_thm_1.1}.

First, we need a vanishing result of \emph{connected} PT-invariants.

\begin{proposition} \label{logPT_zero}
Let $d\geq 1$ be an integer, and $B(-)\colon \mathbb{Z}_{>0}\to \mathbb{R}_{\geq 0}$ be a function such that $g(C)\leq B\big(d(C)\big)$ for any curve $C\subset X$. Assume that 
\begin{equation} \label{bd_2}
    B(d)-1\geq \sum_{i=1}^s \big(B(d_i)-1\big)
\end{equation}
for any integer $s\geq 1$ and $d_i\geq 1$ with $\sum^s_{i=1} d_i=d$.
Then the coefficient of $q^mt^{\beta}$ in $\sF_P(q, t)$ is zero for any effective 1-cycle $\beta\neq 0$ with $\deg_H \beta=d$ and $m<1-B(d)$.

\end{proposition}

\begin{proof}
By definition, we obtain
\[\sF_P(q, t)=\log \PT(q, t)=\sum_{k=1}^{\infty} \frac{(-1)^{k-1} \PT^{'k}(q, t)}{k},\]
where $$\PT^{'}(q, t):=\PT(q, t)-1=\sum_{\beta\neq 0}\sum_{n\in \ZZ} P_{n, \beta} q^n t^{\beta}=\sum_{\beta \in H_2(X, \mathbb{Z})_c}\sum_{n\in \ZZ} P_{n, \beta} q^n t^{\beta}.$$ 
Thus the coefficient of $q^mt^{\beta}$ in $\sF_P(q, t)$ is a $\QQ$-linear combination of integers in the set
\[A_{m,\beta}:=\{\prod_{i=1}^s P_{n_i, \beta_i} \mid s\in \mathbb{Z}_{\geq 1}, n_i\in \ZZ, \beta_i\in H_2(X, \mathbb{Z})_c, n_1+...+n_s=m ~\text{and}~ \beta_1+...+\beta_s=\beta\}.\]

Let $d_i:=\deg_H \beta_i$. By assumptions, we have $n_1+...+n_s=m< 1-B(d)$ and $d_1+...+d_s=d$. From Corollary \ref{PT_zero}, we can assume that $n_i\geq 1-B(d_i)$ for any $i$, otherwise $\prod_{i=1}^s P_{n_i, \beta_i}\in A_{m,\beta}$ is zero and has no contribution to the coefficient. Then we obtain
\[\sum^s_{i=1} \big(B(d_i)-1\big)\geq -(n_1+...+n_s)=-m>B(d)-1,\]
which contradicts \eqref{bd_2}. Thus $A_{m,\beta}=\{0\}$ and the result follows.
\end{proof}

Now we are ready to prove the vanishing of GV-invariants, under some assumptions on $B(d)$.

\begin{proposition}  \label{BPS_zero}
Let $B(-)\colon \mathbb{Z}_{>0}\to \mathbb{R}_{\geq 0}$ be a function such that $g(C)\leq B\big(d(C)\big)$ for any curve $C\subset X$. Assume that 
\begin{equation} \label{5.3_bd}
    B(\sum^s_{i=1} d_i)-1\geq \sum_{i=1}^s \big(B(d_i)-1\big)
\end{equation}
for any integer $s\geq 1$ and $d_i\geq 1$. Moreover, assume that
\begin{equation} \label{bd_1}
    \frac{B(d)-1}{r}+1\geq B(\frac{d}{r})
\end{equation}
holds for any integer $r\geq 1$ and $d\geq 1$.
Then for any effective 1-cycle $\beta\neq 0$,  we have $n_{g}^{\beta}=0$ whenever $g>B(\deg_H \beta)$.
\end{proposition}

\begin{proof}
We prove our proposition by induction on $\deg_H \beta$. When $\deg_H \beta=1$, since in this case $\beta$ is an irreducible class, the result follows from \cite[Theorem 5]{PT10}. 

Assume that $d>1$ and the statement holds for all non-zero effective 1-cycles $\beta'\in H_2(X, \mathbb{Z})$ with $\deg_H \beta'<d$. Now we need to deal with the case $d=\deg_H \beta$. By \cite[Lemma 3.12]{PT07}, we can find an integer $G$ such that $n_G^{\beta}\neq 0$ and  $n_g^{\beta}=0$ for every $g>G$. Thus we only need to prove that $G\leq B(d)$.

Assume otherwise $G>B(d)$, thus $G\geq 1$. By \eqref{GV_PT_5}, the coefficient of $(-q)^{1-G}t^{\beta}$ in $\sF_P(q, t)$ is given by
\[\sum_{g\geq 1, r\in \ZZ_{\geq1}, rl=1-G~ \text{for integers} ~ 1-g\leq l~ \text{and} ~\frac{\beta}{r}\in H_2(X, \ZZ)} n_g^{\frac{\beta}{r}} \frac{(-1)^{g-1}}{r} \binom{2g-2}{l+g-1}.\]
By $1-G\leq 0$ and \eqref{expand}, the genus zero GV-invariants  $n_0^{\frac{\beta}{r}}$ have no contribution to the coefficient of $(-q)^{1-G}t^{\beta}$. Thus the coefficient of $q^{1-G}t^{\beta}$ in $\sF_P(q, t)$ is
\[S_{1-G, \beta}:=\sum_{g\geq 1, r\in \ZZ_{\geq1}, rl=1-G~ \text{for integers} ~ 1-g\leq l~ \text{and} ~\frac{\beta}{r}\in H_2(X, \ZZ)} n_g^{\frac{\beta}{r}} \frac{(-1)^{g-G}}{r} \binom{2g-2}{l+g-1}.\]
By assumption, we have $n_g^{\beta}=0$ for any $g>G$. Thus
\[S_{1-G, \beta}=n^{\beta}_G +\sum_{g\geq 1, rl=1-G~ \text{for integers} ~ 1-g\leq l~ \text{and}~r\geq 2~ \text{with} ~\frac{\beta}{r}\in H_2(X, \ZZ)} n_g^{\frac{\beta}{r}} \frac{(-1)^{g-G}}{r} \binom{2g-2}{l+g-1}.\]

Now we claim that $g>B(\frac{d}{r})$ for any integers $g$ and $r$ satisfying
\[g\geq 1, rl=1-G~ \text{for integers} ~ 1-g\leq l~ \text{and}~ r\geq 2~\text{with} ~\frac{d}{r}\in \ZZ.\]
Indeed, we have
\[g\geq 1-l=\frac{G-1}{r}+1> \frac{B(d)-1}{r}+1\]
by our assumption that $G>B(d)$. Then by \eqref{bd_1}, we get
\[g>B(\frac{d}{r}).\]
Hence from the induction hypothesis, we obtain $S_{1-G, \beta}=n^{\beta}_G$. Thus by our assumption $n^{\beta}_G\neq 0$, we see $S_{1-G, \beta}\neq 0$. But from Proposition \ref{logPT_zero}, we have $S_{1-G, \beta}=0$, which makes a contradiction. This means $G\leq B(d)$ as desired and finishes our proof.
\end{proof}

As an application, we can prove one of our main theorems.

\begin{proof}[Proof of Theorem \ref{main_thm_1.1}] We define 
\[B(d):=\frac{d^2+5d+10}{10}.\]
It is not hard to check that \eqref{5.3_bd} and \eqref{bd_1} hold for  $B(-)$. Then the result follows from Theorem \ref{bound_quintic} and Proposition \ref{BPS_zero}.
\end{proof}

\section{Non-vanishing of Gopakumar--Vafa invariants} \label{sec_nonvanish}

\subsection{Moduli spaces}
In this subsection, we fix $X\subset \PP^4$ to be a smooth hypersurface of degree $n\leq 5$.
Using Theorem \ref{20_51}, we can describe some moduli spaces explicitly.

For $k>0$ and any closed subscheme $S\subset \PP^4$, let $\Hilb_{S}^{P_{\oh_S(k)}}$ be the Hilbert scheme of closed subschemes of $S$ with the same Hilbert polynomial as effective divisors in $|\oh_S(k)|$. For example, we have
\[\Hilb^{P_{\oh_X(k)}}_{X}\cong \PP\mathrm{H}^0\big(X, \oh_{X}(k)\big).\]
For any two polynomials $P,Q \in \QQ[t]$, we denote the flag Hilbert scheme by $\FHilb^{P, Q}_X$. It parameterizes pairs of closed subschemes $[Z\subset Y]$, where $[Z]\in \Hilb^P_X$ and $[Y]\in \Hilb^{Q}_X$. For simplicity, we denote
\[\cM_d:=\Hilb_X^{dt-\frac{1}{2n}d^2-\frac{n-4}{2}d}.\]


\begin{lemma} \label{moduli}
Let $X\subset \PP^4$ be a smooth hypersurface of degree $n\geq 1$. For any $m\geq 1$ and hyperplane section $D\subset X$, we have
\[\Hilb^{P_{\oh_D(m)}}_D\cong \PP\mathrm{H}^0\big(D, \oh_D(m)\big).\]
\end{lemma}

\begin{proof}
Let 
\[f:=\pr_2\circ i\colon \widetilde{C}\xhookrightarrow{i} D\times \PP\mathrm{H}^0\big(D, \oh_D(m)\big)\xra{\pr_2} \PP\mathrm{H}^0\big(D, \oh_D(m)\big)\]
be the universal hypersurface. Note that $f$ is proper, flat, and each fibre over a closed point is a closed subscheme in $D$ with the Hilbert polynomial $P_{\oh_D(m)}$. Then we obtain a morphism
\[q\colon\PP\mathrm{H}^0\big(D, \oh_D(m)\big)\to \Hilb^{P_{\oh_D(m)}}_D.\]
Since fibres of $f$ over distinct points are also distinct, we know that $q$ is a closed immersion. Now by Proposition \ref{prop_D}, $q$ is bijective, and therefore a homeomorphism. Thus we see $\dim \Hilb^{P_{\oh_D(m)}}_D=\dim \PP\mathrm{H}^0\big(D,\oh_D(m)\big)$. 

To prove $q$ is an isomorphism, we only need to show that $\Hilb^{P_{\oh_D(m)}}_D$ is smooth. To this end, the tangent space at any closed point $[C]\in \Hilb^{P_{\oh_D(m)}}_D$ is given by $T_{[C]}\Hilb^{P_{\oh_D(m)}}_D=\mathrm{H}^0\big(C, N_{C/D}\big)$. By Proposition \ref{prop_D}, we have $\mathrm{H}^0\big(C, N_{C/D}\big)=\mathrm{H}^0\big(C, \oh_C(m)\big)$.
Then it is not hard to compute that
\[\dim T_{[C]}\Hilb^{P_{\oh_D(m)}}_D=\dim \mathrm{H}^0\big(D, \oh_D(m)\big)-1= \dim \PP\mathrm{H}^0\big(D, \oh_D(m)\big)=\dim \Hilb^{P_{\oh_D(m)}}_D.\]
Thus $\Hilb^{P_{\oh_D(m)}}_D$ is smooth and the result follows.
\end{proof}

\begin{proposition} \label{moduli_fibre}
Let $X\subset \PP^4$ be a smooth hypersurface of degree $n\leq 5$. Let $m\geq 2$ and $d=nm$ be integers. Then the moduli space $\cM_d$ is smooth and a $\PP^{\dim \mathrm{H}^0\big(D, \oh_D(m)\big)-1}$-bundle over $\Hilb_X^{P_{\oh_X(1)}}$. Here $D\subset X$ is any hyperplane section.
\end{proposition}

\begin{proof}
Note that $\cM_d$ is smooth and isomorphic to the flag Hilbert scheme \[\cM_d\cong \FHilb^{dt-\frac{1}{2n}d^2-\frac{n-4}{2}d,P_{\oh_X(1)}}_X.\] 
Indeed, every curve $[C]\in \cM_d$ is a $(1,\frac{d}{n},n)$-complete intersection in $\PP^4$ by Theorem \ref{20_51}, then the smoothness follows from a direct computation of obstruction spaces. And $\FHilb^{dt-\frac{1}{2n}d^2-\frac{n-4}{2}d,P_{\oh_X(1)}}_X$ is constructed as the projective bundle $\PP\big(\pr_{1*}(\mathcal{I}(1))\big)$ over $\cM_d$, where $\mathcal{I}$ is the universal ideal sheaf on $\cM_d\times X$ and $\pr_1\colon \cM_d\times X\to \cM_d$ is the projection. By the uniqueness part of Lemma \ref{extremal_curve}, we know that $\pr_{1*}(\mathcal{I}(1))$ is a line bundle on $\cM_d$. Then we conclude that $\cM_d\cong \FHilb^{dt-\frac{1}{2n}d^2-\frac{n-4}{2}d,P_{\oh_X(1)}}_X$.

Now we have a natural morphism
\[p\colon \cM_d\cong \FHilb^{dt-\frac{1}{2n}d^2-\frac{n-4}{2}d,P_{\oh_X(1)}}_X\hookrightarrow \cM_d\times \Hilb^{P_{\oh_X(1)}}_X\to \Hilb^{P_{\oh_X(1)}}_X\]
given by $[C]\mapsto [C\subset D]\mapsto [C]\times [D]\mapsto [D]$, 
with fibres $p^{-1}([D])\cong \Hilb_D^{P_{\oh_D(m)}}$ over points $[D]\in \Hilb^{P_{\oh_X(1)}}_X$. Then the result follows from Lemma \ref{moduli}.
\end{proof}

\subsection{Non-vanishing}

In this subsection, we fix $X$ to be a quintic 3-fold. We are going to prove Theorem \ref{non_vanish}. Let
\[B(d):=\frac{d^2+5d+10}{10}.\]

For any two integers $d\geq 0$ and $n$, the Hilbert scheme $\Hilb^{dt+n}_X$ can be regarded as the moduli space of torsion-free sheaves on $X$ with Chern character \[\big(1,0,-\frac{d}{5}H^2,-\frac{n}{5}H^3\big).\]
In \cite{thomasDT}, Thomas constructed a zero-dimensional virtual cycle $[\Hilb^{dt+n}_X]^{vir}$ on the moduli space $\Hilb^{dt+n}_X$. Then we define a \emph{Donaldson--Thomas (DT) invariant} by
\[I_{n, d}:=\int_{[\Hilb^{dt+n}_X]^{vir}} 1.\]
When $\Hilb^{dt+n}_X$ is smooth, we have
\[I_{n,d}=(-1)^{\dim \Hilb^{dt+n}_X}e(\Hilb^{dt+n}_X),\]
where $e(-)$ denotes the topological Euler characteristic.

A relation between PT-invariants and DT-invariants is conjectured by Pandharipande and Thomas \cite[Conjecture 3.3]{PT07} and proved by Bridgeland \cite{BriHall}. Let 
\[\DT_d(q):=\sum_{n} I_{n, d}q^n\]
be the generating series of degree $d$  DT-invariants. By \cite[Theorem 1.1]{BriHall}, we obtain
\[\PT_d(q)\cdot \DT_0(q)=\DT_d(q).\]
Note that $I_{n,0}=0$ for any $n<0$. Then for any $n$ and $d>0$, we have
\[\sum_{m\geq 0} P_{n-m, d}\cdot I_{m, 0}=I_{n, d}.\]
In particular, if $n=1-B(d)$, by Corollary \ref{PT_zero} we obtain
\begin{equation} \label{PT=DT}
    P_{1-B(d),d}=P_{1-B(d),d}\cdot I_{0,0}=I_{1-B(d),d}.
\end{equation}

\begin{lemma} \label{PT_lem}
Let $d\in 5\ZZ_{\geq 1}$ be an integer.
Then the coefficient of $q^{1-B(d)}t^{d}$ in $\sF_P(q, t)$ is equal to $P_{1-B(d),d}$.
\end{lemma}

\begin{proof}
By a similar argument in Proposition \ref{logPT_zero}, we know that the coefficient of $q^{1-B(d)}t^{d}$ in $\sF_P(q, t)$ is a $\QQ$-linear combination of integers in the set
\[A_{1-B(d),d}=\{\prod_{i=1}^s P_{n_i, d_i} \mid s,d_i\in \mathbb{Z}_{\geq 1}, n_i\in \ZZ, n_1+...+n_s=1-B(d) ~\text{and}~ d_1+...+d_s=d\}.\]
Moreover, we can assume that $n_i\geq 1-B(d_i)$ for any $i$. Then we obtain
\[\sum^s_{i=1} \big(B(d_i)-1\big)\geq -(n_1+...+n_s)=B(d)-1.\]
Since \eqref{bd_2} in Proposition \ref{logPT_zero} holds for $B(d)$, we have
\[\sum^s_{i=1} \big(B(d_i)-1\big)=-(n_1+...+n_s)=B(d)-1.\]
Note that in our case $B(d)=\frac{d^2+5d+10}{10}$, it is not hard to check that the only possibility is $s=1$. Then we see
\[A_{1-B(d), d}=\{P_{1-B(d), d}\}\cup \{0\}\]
and the result follows.
\end{proof}

\begin{lemma} \label{BPS_lem}
Let $d\in 5\ZZ_{\geq 1}$ be an integer.
Then we have 
\[n^{d}_{B(d)}=P_{1-B(d),d}=I_{1-B(d),d}.\]
\end{lemma}

\begin{proof}
From the proof of Proposition \ref{BPS_zero}, the coefficient $S_{1-B(d),d}$ of $q^{1-B(d)}t^{d}$ in $\sF_P(q, t)$ is given by
\[S_{1-B(d), d}=n^d_{B(d)} +\sum_{g\geq 1, rl=1-B(d), ~ 1-g\leq l~ \text{and}~r\geq 2~ \text{with} ~\frac{d}{r}\in \ZZ} n_g^{\frac{d}{r}} \frac{(-1)^{g-B(d)}}{r} \binom{2g-2}{l+g-1}.\]
And we have
\[g\geq -l+1\geq \frac{B(d)-1}{r}+1.\]
Since $B(d)=\frac{d^2+5d+10}{10}$, one can check
\[g\geq \frac{B(d)-1}{r}+1>B(\frac{d}{r})\]
when $r\geq 2$, which gives $S_{1-B(d),d}=n^{d}_{B(d)}$ by Theorem \ref{main_thm_1.1}. Then the result follows from Lemma \ref{PT_lem} and \eqref{PT=DT}.
\end{proof}

\begin{proposition} \label{DT_nonvanish}
Let $m\geq 2$ and $d=5m$ be integers. Then we have
\[I_{-\frac{d^2+5d}{10},d}=(-1)^{\binom{m+3}{3}-\binom{m-2}{3}+3}\cdot 5\big(\binom{m+3}{3}-\binom{m-2}{3}\big).\]
Here if $m-2<3$, we define $\binom{m-2}{3}:=0$.
\end{proposition}

\begin{proof}
By Proposition \ref{moduli_fibre}, we obtain
\[(-1)^{-\dim \cM_d}I_{-\frac{d^2+5d}{10},d}=e(\cM_d)=e(\PP^4)e(\PP^{\dim \mathrm{H}^0\big(D, \oh_D(m)\big)-1})=5\dim \mathrm{H}^0\big(D, \oh_D(m)\big),\]
where $D\subset X$ is any hyperplane section. Then the result follows from a direct computation of $\dim \cM_d$ and $\dim \mathrm{H}^0\big(D, \oh_D(m)\big)$.
\end{proof}

Now we are ready to prove one of our main results.

\begin{proof}[Proof of Theorem \ref{non_vanish}]
The result follows from Lemma \ref{BPS_lem} and Proposition \ref{DT_nonvanish}.
\end{proof}

\begin{remark} \label{m=1}
For the remaining case $m=1$, a similar argument as above also works. One can show that any curve of degree $5$ and genus $6$ in $X$ lies on some hyperplane sections. Then Theorem \ref{20_51} also holds in this case. Therefore, the same argument in Proposition \ref{moduli_fibre} shows that the flag Hilbert scheme $\FHilb_X^{5t-5, p_{\oh_X(1)}}$ is a $\PP^1$-bundle over $\cM_5$, and is also a $\PP^3$-bundle over $\PP^4$. Thus we can compute that $I_{-5,5}=n^5_6=10$. 
\end{remark}

\begin{appendix}

\section{Computing GW-invariants via BCOV axioms and Castelnuovo bound (by Shuai Guo)} \label{sec_appendix}

As Klemm and his collaborators showed, BCOV axioms + Castelnuovo bound compute GW-invariants of quintic 3-folds up to $g=53$ \cite{HKQ09}. To help the readers in mathematics, we explain the main reason behind their argument in the appendix. Readers are encouraged to consult their paper \cite{HKQ09} for more details.

A simple investigation of GV-formula \eqref{GW=GV} tells that when $1\leq d$, the invariant $N_{g,d}$ is computed from $n_{g'}^{d'}$ for $g'\leq g,1\leq d'\leq d$ with $n_g^d$ as the leading term. This is the situation that there is an upper triangular linear transformation from $\{n_g^d\}_{0\leq g,1\leq d}$ to $\{N_{g,d}\}_{0\leq g,1\leq d}$. One can invert it
to compute $n_{g}^{d}$ from $N_{g', d'}$ for $g'\leq g,1\leq d'\leq d$.

Now assume that there is a function
\[D(-)\colon \ZZ_{\geq 1}\to \ZZ\]
such that $n_g^{d}=0$ for any $d\leq D(g)$. Then for any two integers $1\leq d$ and $0\leq g$ satisfies $d\leq D(g)$, one can compute $N_{g,d}$ from $\{n_{g'}^{d'}\}_{g'<g, 1\leq d'\leq d}$.

More explicitly, for $g\leq 53, g\neq 51, 1\leq d\leq \lfloor\frac{2g-2}{5}\rfloor$ and $g=51, 1\leq d<20$, the Castelnuovo bound (Corollary \ref{cast_bound}) yields that
$N_{g,d}$ is computed from $n^{g'}_{d'}$ for $g'<g, d'\leq d$.  For $N_{51, 20}$, we add an additional contribution $n_{51}^{20}$, which is $n^{20}_{51}=175$ by Theorem \ref{non_vanish}.  Then, we compute $n^{g'}_{d'}$ using the corresponding $\{N_{g'', d''}\}_{g''\leq g', 1\leq d''\leq d'}$. As a consequence, for $g\leq 53$ and $1\leq d\leq \lfloor\frac{2g-2}{5}\rfloor$, the invariant $N_{g,d}$ is computed from $N_{g',d'}$ for $g'<g, 1\leq d'\leq d$ and $n_{51}^{20}$.

Now, we are back to Klemm's argument. Applying finite generation and holomorphic anomaly equation inductively on genus $g$, 
we can write
$$F_g=  Y^{-(g-1)}   f_g+ \text{known terms} ,\qquad f_g=  \sum^{3g-3}_{i = 0} a_{i,g} Y^i .$$
For quintic 3-folds, it is known that 
$$ Y=(1-5^5\cdot q)^{-1}$$
where $q$ is the so-called B-model coordinate at large complex structure limit $q=0$. It is related to A-model coordinate $t$ via the mirror map $q=q(t)$.

We aim to solve the initial data inductively by genus $g$.

When $g\leq 1$, the result is known by \cite{Givental1998,LLY97,Zi08}. For a given $g>1$, assuming we know $F_{h}$ for $h<g$.
Then the initial data can be fixed via
\begin{itemize}
\item Step 1.  Orbifold regularity condition, $a_{i, g} =0$ for $i\leq \lceil\frac{3g-3}{5}\rceil  $ ;
\item Step 2.  Conifold gap condition, solve $a_{i,g}$ for $i>g-1$;
\item Step 3.  The degree $0$ GW-invariants $N_{g,0}$ are computable;
\item Step 4.  Castelnuovo bound:  fix $N_{g,d}$ via $\{N_{h,d} \}_{d> 0,h<g}$ for $d\leq \lfloor\frac{2g-2}{5}\rfloor$.
\end{itemize}

\subsection{Conifold gap condition}
The condition gives
$$
\sum^{g-1}_{k = 0} a_{g-1-k,g} Y^{-k}  +\sum^{2g-2}_{i = 1} a_{i+g-1,g} Y^{i}  + \text{known terms} = F_g = \frac{(-1)^{g-1} B_{2g}}{2g(2g-2)\Delta^{2g-2}}+O(\Delta^0),
$$
where $\Delta$ is the so-called flat coordinate at conifold $q=5^{-5}$ and $B_{2g}$ is the Bernoulli number.
It is known that 
$$
Y^{-1} = \frac{\delta}{1+\delta}=\Delta+O(\Delta^2) $$
where $\delta=1-5^{-5}q$. For $k\geq 0$ we have
$$Y^{-k}= O(\Delta^0), \qquad  \Delta \rightarrow 0.
$$
Hence the gap condition becomes

\begin{equation}  \label{conigap}
\sum^{2g-2}_{i = 1} a_{i+g-1,g} Y^{i}  + \text{known terms} = \frac{(-1)^{g-1} B_{2g}}{2g(2g-2)\Delta^{2g-2}} +O(\Delta^0), \qquad  \Delta \rightarrow 0 .
\end{equation}
Since  $\Delta =\delta+O(\delta^2)$,  we have
$$Y = 1+\frac{1}{\delta} = \frac{1}{\Delta} +O(\Delta^0), \qquad  \Delta \rightarrow 0.
$$ 
It is clear that  the restriction of
$$
Y,Y^2,\cdots,Y^{2g-2}  \in \Delta^{-(2g-2)}\mathbb Q[[\Delta]]
$$ 
to the space 
$$
\Delta^{-(2g-2)}\mathbb Q[\Delta]_{\deg\leq 2g-2}
$$
are of the form
$$
Y^k = \frac{1}{\Delta^k} \cdot \Big(1+ O(\Delta) \Big).
$$
Especially, they are linearly independent. By looking at the coefficient of 
$$\Delta^{-1}, \Delta^{-2},  \cdots,  \Delta^{-(2g-2)},$$
 the equation \eqref{conigap} will give a upper triangular linear equations in $\{a_{k+g-1}\}_{k=1}^{2g-2}$.
Hence the $\{a_{k+g-1}\}_{k=1}^{2g-2}$ can be solved uniquely via \eqref{conigap}.

\subsection{Degree $0$ invariants + Castelnuovo bound}
Now we have
$$
\sum^{g-1}_{k = 0} a_{g-1-k,g} Y^{-k}    + \text{known terms} =F_g = \sum_{d\geq 0} N_{g,d} t^d .$$
If $n_g^{d}=0$ for $d\leq D(g)$, one can compute $N_{g,d}$ from lower genus invariants via Gopakumar--Vafa formula, as the argument at the beginning shows. Combining with the induction hypothesis and the degree zero GW-invariants, this means the following is known:
$$
\sum_{d=0}^{D(g)} N_{g,d} t^d
$$
Together with the orbifold regularity and the fact that $t=q+O(q^2)$, we see
\begin{align*}
\sum^{\lfloor\frac{2(g-1)}{5} \rfloor}_{k = 0} a_{g-1-k,g} (1- 5^{5}\cdot q)^{k}    =\ &  \text{known  series in $t$} +O(t^{D(g)+1})\\ =\  &  \text{known polynomial in  $q$} +O(q^{D(g)+1})
\end{align*}
Note that the coefficients up to $q^{D(g)}$ of the inverse mirror map will enter the right-hand side of the above equation.
For any $D\geq 0$,
$$
\{  (1-5^5\cdot q)^{k} \}_{k=0}^{D}  \subset   \mathbb Q[q]_{\deg\leq D}
$$
are linearly independent. Therefore, if \[E:=\min\{D(g),\lfloor\frac{2(g-1)}{5} \rfloor\},\]
by looking at the coefficients of 
$$q^0, q^1,  \cdots,  q^{E},$$
in the above equation, we will obtain a lower triangular linear equations in $\{ a_{g-1-k,g} \}_{k=0}^{E}$.
Hence $\{ a_{g-1-k,g} \}_{k=0}^{E}$
can be uniquely solved.
Especially, if
$$
D(g) \geq \lfloor\frac{2(g-1)}{5} \rfloor , 
$$
all the $\{a_{j,g}\}_{j=0}^{3g-3}$ are uniquely solved by arguments above.

In particular, since BCOV axioms are proved by \cite{GJR18, CGL18, CGLFeynman} except the conifold gap condition, combining with Theorem \ref{main_thm_1.1} and Theorem \ref{non_vanish} we have:

\begin{theorem} \label{appendix_thm}
Let $G\leq 53$ be a positive integer. Assume that the conifold gap condition holds for $F_g(t)$ and any $g\leq G$. Then we can calculate $F_g(t)$ effectively up to genus $G$. 
\end{theorem}

\end{appendix}

\bibliographystyle{plain}
{\small{\bibliography{quintic}}}

\begin{thebibliography}{10}

\bibitem{bayer2020desingularization}
A.~Bayer, S.~Beentjes, S.~Feyzbakhsh, G.~Hein, D.~Martinelli, F.~Rezaee, and
  B.~Schmidt.
\newblock The desingularization of the theta divisor of a cubic threefold as a
  moduli space.
\newblock {\em arXiv preprint, arXiv: 2011.12240}, 2020.

\bibitem{bayer2016space}
A.~Bayer, E.~Macr{\`\i}, and P.~Stellari.
\newblock {The space of stability conditions on abelian threefolds, and on some
  Calabi--Yau threefolds}.
\newblock {\em Inventiones mathematicae}, 206(3):869--933, 2016.

\bibitem{bayer2011bridgeland}
A.~Bayer, E.~Macr{\`\i}, and Y.~Toda.
\newblock {Bridgeland Stability conditions on threefolds I: Bogomolov--Gieseker
  type inequalities}.
\newblock {\em Journal of Algebraic Geometry}, 23, 03 2011.

\bibitem{BM22}
A.~Bayer and E.~Macrì.
\newblock {The unreasonable effectiveness of wall-crossing in algebraic
  geometry}.
\newblock {\em arXiv preprint, arXiv:2201.03654}, 2022.

\bibitem{bridgeland}
T.~Bridgeland.
\newblock {Stability conditions on triangulated categories}.
\newblock {\em Annals of Mathematics}, 166(2):317--345, 2007.

\bibitem{BriHall}
T.~Bridgeland.
\newblock {Hall algebras and curve-counting invariants}.
\newblock {\em Journal of the American Mathematical Society}, 24, 02 2010.

\bibitem{CGLFeynman}
H.~Chang, S.~Guo, and J.~Li.
\newblock {BCOV's Feynman rule of quintic 3-folds}.
\newblock {\em arXiv preprint, arXiv: 1810.00394}, 2018.

\bibitem{CGL18}
H.~Chang, S.~Guo, and J.~Li.
\newblock {Polynomial structure of Gromov--Witten potential of quintic
  3-folds}.
\newblock {\em Annals of Mathematics}, 194(3):585 -- 645, 2021.

\bibitem{Chang_2021}
H.~Chang, S.~Guo, J.~Li, and W.~Li.
\newblock {The theory of N-mixed-spin-P fields}.
\newblock {\em Geometry\&Topology}, 25(2):775--811, 2021.

\bibitem{Chen_2021}
Q.~Chen, F.~Janda, and Y.~Ruan.
\newblock {The logarithmic gauged linear sigma model}.
\newblock {\em Inventiones mathematicae}, 225(3):1077--1154, 2021.

\bibitem{CJR22}
Q.~Chen, F.~Janda, and Y.~Ruan.
\newblock {Punctured logarithmic R-maps}.
\newblock {\em arXiv preprint, arXiv: 2208.04519}, 2022.

\bibitem{CJRS18}
Q.~Chen, F.~Janda, Y.~Ruan, and A.~Sauvaget.
\newblock {Towards a Theory of Logarithmic GLSM Moduli Spaces}.
\newblock {\em arXiv preprint, arXiv: 1805.02304}, 2018.

\bibitem{DIW21}
A.~Doan, E.-N. Ionel, and T.~Walpuski.
\newblock {The Gopakumar-Vafa finiteness conjecture}.
\newblock {\em arXiv preprint, arXiv:2103.08221}, 2021.

\bibitem{FanLee19}
H.~Fan and Y.~P. Lee.
\newblock {Towards a quantum Lefschetz hyperplane theorem in all genera}.
\newblock {\em Geometry\&Topology}, 23(1):493--512, 2019.

\bibitem{F2022}
S.~Feyzbakhsh.
\newblock {Rank zero DT invariants}.
\newblock {\em In preparation}, 2022.

\bibitem{Givental1998}
A.~Givental.
\newblock {\em A mirror theorem for toric complete intersections}, pages
  141--175.
\newblock Birkh{\"a}user Boston, Boston, MA, 1998.

\bibitem{GV1}
R.~Gopakumar and C.~Vafa.
\newblock {M-Theory and Topological Strings--I}.
\newblock {\em arXiv preprint, arXiv: hep-th/9809187}, 1998.

\bibitem{GV2}
R.~Gopakumar and C.~Vafa.
\newblock {M-Theory and Topological Strings--II}.
\newblock {\em arXiv preprint, arXiv: hep-th/9812127}, 1998.

\bibitem{GJR17P}
S.~Guo, F.~Janda, and Y.~Ruan.
\newblock {A mirror theorem for genus two Gromov--Witten invariants of quintic
  threefolds}.
\newblock {\em arXiv preprint, arXiv:1709.07392}, 2017.

\bibitem{GJR18}
S.~Guo, F.~Janda, and Y.~Ruan.
\newblock {Structure of higher genus Gromov--Witten invariants of quintic
  3-folds}.
\newblock {\em arXiv preprint, arXiv:1812.11908}, 2018.

\bibitem{happel1996tilting}
D.~Happel, I.~Reiten, and S.~O. Smal{\o}.
\newblock {\em Tilting in abelian categories and quasitilted algebras}, volume
  575.
\newblock American Mathematical Soc., 1996.

\bibitem{harris1980}
J.~Harris.
\newblock {The genus of space curves}.
\newblock {\em Mathematische Annalen}, 249:191--204, 1980.

\bibitem{har80}
R.~Hartshorne.
\newblock Stable reflexive sheaves.
\newblock {\em Mathematische Annalen}, 254:121--176, 1980.

\bibitem{Hartshorne1994TheGO}
R.~Hartshorne.
\newblock The genus of space curves.
\newblock {\em Annali dell’Universit{\`a} di Ferrara}, 40:207--223, 1994.

\bibitem{PhysRevD.79.066001}
M.-x. Huang, A.~Klemm, M.~Mari\~no, and A.~Tavanfar.
\newblock {Black holes and large order quantum geometry}.
\newblock {\em Phys. Rev. D}, 79:066001, 2009.

\bibitem{HKQ09}
M.-x. Huang, A.~Klemm, and S.~Quackenbush.
\newblock {\em Topological String Theory on Compact Calabi--Yau: Modularity and
  Boundary Conditions}, pages 1--58.
\newblock Springer Berlin Heidelberg, Berlin, Heidelberg, 2009.

\bibitem{IP13}
E.-N. Ionel and T.~Parker.
\newblock {The Gopakumar--Vafa formula for symplectic manifolds}.
\newblock {\em Annals of Mathematics}, 187, 06 2013.

\bibitem{KKV}
S.~Katz, A.~Klemm, and C.~Vafa.
\newblock {M-Theory, Topological Strings and Spinning Black Holes}.
\newblock {\em Adv. Theor. Math. Phys.}, 3, 11 1999.

\bibitem{koseki2020bogomolov}
N.~Koseki.
\newblock {On the Bogomolov--Gieseker inequality for hypersurfaces in the
  projective spaces}.
\newblock {\em arXiv preprint, arXiv:2008.09799}, 2020.

\bibitem{lang04}
A.~Langer.
\newblock {Semistable sheaves in positive characteristic}.
\newblock {\em Annals of Mathematics}, 159(1):251--276, 2004.

\bibitem{li2018stability}
C.~Li.
\newblock {Stability conditions on Fano threefolds of Picard number 1}.
\newblock {\em Journal of the European Mathematical Society}, 21(3):709--726,
  2018.

\bibitem{liquintic}
C.~Li.
\newblock On stability conditions for the quintic threefold.
\newblock {\em Inventiones mathematicae}, 218, 10 2019.

\bibitem{LLY97}
B.~H. Lian, K.~Liu, and S.-T. Yau.
\newblock {Mirror principle. I}.
\newblock {\em Asian Journal of Mathematics}, 1(4):729--763, 1997.

\bibitem{MP2}
A.~Maciocia and D.~Piyaratne.
\newblock {Fourier-Mukai transforms and Bridgeland stability conditions on
  abelian threefolds II}.
\newblock {\em International Journal of Mathematics}, 27, 10 2013.

\bibitem{MP1}
A.~Maciocia and D.~Piyaratne.
\newblock {Fourier-Mukai transforms and Bridgeland stability conditions on
  abelian threefolds}.
\newblock {\em Algebraic Geometry}, 2 (3):270--297, 2015.

\bibitem{Mac14}
E.~Macr{\`i}.
\newblock {A generalized Bogomolov--Gieseker inequality for the
  three-dimensional projective space}.
\newblock {\em Algebra \& Number Theory}, 8(1):173 -- 190, 2014.

\bibitem{MS20}
E.~Macrì and B.~Schmidt.
\newblock Derived categories and the genus of space curves.
\newblock {\em Algebraic Geometry}, 01 2020.

\bibitem{MP06}
D.~Maulik and R.~Pandharipande.
\newblock {A topological view of Gromov--Witten theory}.
\newblock {\em Topology}, 45:887--918, 09 2006.

\bibitem{PP12}
R.~Pandharipande and A.~Pixton.
\newblock {Gromov--Witten/Pairs correspondence for the quintic 3-fold}.
\newblock {\em Journal of the American Mathematical Society}, 30, 06 2012.

\bibitem{PT07}
R.~Pandharipande and R.~Thomas.
\newblock {Curve counting via stable pairs in the derived category}.
\newblock {\em Inventiones mathematicae}, 178, 07 2007.

\bibitem{PT10}
R.~Pandharipande and R.~Thomas.
\newblock {Stable pairs and BPS invariants}.
\newblock {\em Journal of the American Mathematical Society}, 23(1):267--297,
  2010.

\bibitem{sanna2014rational}
G.~Sanna.
\newblock {Rational curves and instantons on the Fano threefold $Y_5$}.
\newblock {\em arXiv preprint, arXiv:1411.7994}, 2014.

\bibitem{sch13}
B.~Schmidt.
\newblock {A generalized Bogomolov--Gieseker inequality for the smooth quadric
  threefold}.
\newblock {\em Bulletin of the London Mathematical Society}, 46(5):915--923, 06
  2014.

\bibitem{thomasDT}
R.~Thomas.
\newblock {A holomorphic Casson invariant for Calabi--Yau 3-folds, and bundles
  on K3 fibrations}.
\newblock {\em Journal of Differential Geometry}, 54, 12 2000.

\bibitem{Zi08}
A.~Zinger.
\newblock {The reduced genus-one Gromov--Witten invariants of Calabi--Yau
  hypersurfaces}.
\newblock {\em Journal of the American Mathematical Society}, 22:691--737,
  2009.

\end{thebibliography}

\end{document}